\documentclass{siamart190516}

\usepackage{amsfonts}
\usepackage{multirow}
\usepackage{float}
\usepackage{array}
\usepackage{textcomp}
\usepackage{amsmath}
\usepackage{amssymb}
\usepackage{algorithmic}
\usepackage{mathtools}
\usepackage{enumitem}
\usepackage{color}
\usepackage{booktabs}
\usepackage{array}
\usepackage{appendix}
\usepackage{tablefootnote}

\definecolor{green}{rgb}{0,0.6,0.0} 

\global\long\def\rn{\Re^{n}}%
\global\long\def\r{\Re}%
\allowdisplaybreaks
\renewcommand{\tabcolsep}{2.5pt}

\usepackage{mathtools}

\global\long\def\argmin{\operatorname*{argmin}}%

\newcommand{\R}{\mathbb{R}}

\newcommand{\inner}[2]{\langle{#1},{#2}\rangle}
\newcommand{\cl}{\mathrm{cl}\,}

\newcommand{\Inner}[2]{\left\langle{#1},{#2}\right\rangle}
\newcommand{\norm}[1]{\|#1\|}

\newcommand{\ri}{{\mbox{\rm ri\,}}}

\newcommand{\inte}{\mathrm{int}}

\newcommand{\dom}{{\mbox{\rm dom\,}}}
\newcommand{\Gr}{\mathrm{Gr}\,}
\newcommand{\rbd}{\mathrm{rbd}\,}

\newcommand{\vgap}{\vspace{.1in}}

\newcommand{\lam}{\lambda}

\newcommand{\bConv}[1]{\overline{\mbox{\rm Conv}}\,(\Re^{#1})}

\newcommand{\bi}{\begin{itemize}}
\newcommand{\ei}{\end{itemize}}
\newcommand{\ba}{\begin{array}}
\newcommand{\ea}{\end{array}}

\newcommand{\lm}{p}
\newcommand{\pen}{c}
\newcommand{\z}{z}

\usepackage[utf8]{inputenc}
\usepackage[english]{babel}

\date{January ??, 2022}

\title{Iteration-complexity  of an inner  accelerated  inexact proximal augmented Lagrangian method based on the classical Lagrangian function}

\begin{document}
	
	\makeatletter
	\maketitle
	
	\begin{center}
			\textsc{Weiwei Kong}\footnote{Computer Science and Mathematics Division, Oak Ridge National Laboratory,
            Oak Ridge, TN, 37830. (email: {\tt wwkong92@gmail.com}). This author has been supported by (i) the US Department of Energy
            (DOE) and UT-Battelle, LLC, under contract DE-AC05-00OR22725 and (ii) the Exascale Computing Project (17-SC-20-SC), a collaborative
            effort of the U.S. Department of Energy Office of Science and the
            National Nuclear Security Administration.},
			\textsc{Jefferson G. Melo\footnote{Instituto de Matem\'atica e Estat\'istica, Universidade Federal de Goi\'as, Campus II- Caixa
				Postal 131, CEP 74001-970, Goi\^ania-GO, Brazil. (email: {\tt jefferson@ufg.br}). This author was partially supported by CNPq grant 312559/2019-4 and FAPEG/GO. }, 
			\textsc{And} 
			Renato D.C. Monteiro}\footnote{School of Industrial and Systems
			Engineering, Georgia Institute of
			Technology, Atlanta, GA, 30332-0205.
			(email: {\tt monteiro@isye.gatech.edu}). This author
			was partially supported by ONR Grant
			N00014-18-1-2077 and AFORS Grant FA9550-22-1-0088.}
		\end{center}
	
\vspace*{1em}	
	
	\begin{abstract}
This paper establishes the iteration-complexity
of an inner accelerated inexact proximal augmented Lagrangian
(IAIPAL) method for solving linearly-constrained smooth nonconvex composite optimization problems that is
based on the classical augmented Lagrangian (AL) function. 
			More specifically,
			each IAIPAL iteration consists of inexactly solving a proximal AL subproblem by an accelerated composite gradient (ACG) method 
			followed by a classical Lagrange multiplier update.
			Under the assumption that the domain of the composite function is bounded and the problem has a Slater point,
			it is shown that
			IAIPAL generates an approximate  stationary solution 
			in 
			${\cal O}(\varepsilon^{-5/2}\log^2 \varepsilon^{-1})$ ACG iterations where $\varepsilon>0$ is a
			tolerance for both stationarity and feasibility. 
    	Moreover, the above bound is derived without assuming that
    	the initial point is feasible. Finally, numerical results are  presented to demonstrate the strong practical performance of IAIPAL.
    	

\vspace*{1em}


\vspace*{1em}

\end{abstract}

	\section{Introduction} \label{sec:int}

	This paper presents an inner accelerated inexact proximal   augmented Lagrangian
	(IAIPAL) method for solving
	the linearly-constrained smooth nonconvex composite optimization (NCO) problem
	\begin{equation} \label{optl0}
	\phi^*:=\min \{ \phi(z) :=  f(z) + h(z) : A z  =b\},
	\end{equation}
	where  $A:\Re^n \to \Re^l$ is a linear operator, $b\in\Re^l$, $h:\Re^n \to (-\infty,\infty]$ is a closed proper convex function which is $M_h$-Lipschitz continuous on its domain, and $f$ is a real-valued differentiable
	nonconvex function such that, for some scalars $L_f\geq m_f>0$, $f$ is
	$m_f$-weakly convex on the domain,
	$\dom h$, of $h$
	(i.e., satisfies \eqref{lowerCurvature-m} below) and its gradient is $L_f$--Lipschitz.
For a given tolerance pair $(\hat \rho,\hat \eta) \in \Re^2_{++}$,
its goal is to find a triple $(\hat z,\hat p,\hat w)$
 satisfying
\begin{align}\label{eq:approx_stationary1'}
		\hat  w\in\nabla f(\hat z)+\partial h(\hat  z)+A^*\hat  \lm,\quad
	\|\hat w\|\leq \hat \rho,\quad  \|A\hat  z-b\|\leq \hat \eta.
		\end{align}
 More specifically, IAIPAL 
is based on the
	augmented Lagrangian (AL) function ${\cal L}_c(z;p)$ defined as
	\begin{equation}\label{lagrangian2}
	{\cal L}_c(z;\lm):=f(z)+h(z)+\left\langle\lm,Az-b\right\rangle+\frac{\pen}{2}\|Az-b\|^2,
	\end{equation}
	 which has been thoroughly studied
	in the literature (see for example \cite{AybatAugLag,Ber1,LanMonteiroAugLag,ShiqiaMaAugLag16,MR0418919}).
	Roughly speaking, for a fixed
	stepsize $\lam>0$, a scalar $\alpha > 0$, and
	initial points $z_0 \in \dom h$ and
	$p_0=0$,
	IAIPAL repeatedly
	performs the following iteration:
	given $(z_{k-1},\lm_{k-1})  \in \dom h \times \Re^l$, it computes $(z_k,\lm_{k})$ as
	\begin{align}
	z_k &\approx \argmin_z \left \{ \lam {\cal L}_\pen(z,\lm_{k-1}) + \frac12 \| z-z_{k-1}\|^2 \right\} \quad \label{exactzk0}, \\
	\lm_k &
	=  p_{k-1} + 
	\begin{cases}
    c(Az_{k}-b), & k\equiv1\bmod \lceil\alpha c\rceil,\\
    0, & \text{otherwise},
    \end{cases}\label{def:pk0}
	\end{align}
	where $z_k$ in \eqref{exactzk0} is
	a suitable approximate solution of the underlying
	prox-AL subproblem \eqref{exactzk0}. 
IAIPAL sets $\lam=1/(2m_f)$ which,
due to the fact that $f$ is $m_f$-weakly convex, guarantees that the objective 	function of  \eqref{exactzk0} is
	strongly convex. 
	Moreover, it
	computes $z_k$ by approximately solving
	subproblem \eqref{exactzk0} by a
	strongly convex version of FISTA, which is
	a well-known accelerated composite gradient (ACG) variant for solving convex composite
	 optimization problems
	(see for example
	\cite{ beck2009fast,MontSvaiter_fista,nesterov1983method}). The latter point
	is then used to construct a
	triple $(\hat z_k,\hat p_k,\hat w_k)$ and IAIPAL
	stops if it  satisfies \eqref{eq:approx_stationary1'}.
	Otherwise, an auxiliary novel
	test is performed to decide
	whether: i) $c$ should be left unchanged, or; ii)  $c$ is
	updated
	as $c\leftarrow \tau c$ with $\tau>1$ and
	$(z_k,p_k)$ either changed to
	$(z_0,p_0)$ (cold restart)
	or to $(z_k,p_0)$ (warm restart).
	Finally, $k$ is updated to $k+1$ and
	the iteration described above
	is repeated.\\

\noindent\textit{Related works}. 
We mainly focus our attention on works dealing with iteration
complexities of penalty-based and/or AL-based methods. Furthermore, all of the AL methods below use the multiplier update
\begin{gather}
p_{k}= (1-\theta)(p_{k-1}+\chi_{k}c_{k}(Az_k - b)) \label{eq:gen_dual_update2}
\end{gather}
for $\theta\in[0,1)$ and $\chi_k \in [0,1]$ at every $k\geq 1$.
For consistency, the complexities in this review refer to the effort
of obtaining 
an approximate stationary point as in \eqref{eq:approx_stationary1'}.
Note that even though these complexities
are described as bounds on the number of
(possibly ACG) iterations,
they are also 
bounds on the total 
number of $h$-resolvent computations
and/or gradient evaluations of $f$. 

Iteration complexities of quadratic penalty methods for solving  \eqref{optl0} under the assumption that
$f$ is convex and  $h$  is an indicator function of a convex set were first analyzed in \cite{LanRen2013PenMet} and further studied  in
\cite{Aybatpenalty,IterComplConicprog}. Iteration complexities of first-order augmented Lagrangian (AL) methods for solving
the aforementioned class of convex problem have been studied in  \cite{AybatAugLag,LanMonteiroAugLag,Li-Qu19,ShiqiaMaAugLag16,zhaosongAugLag18,Patrascu2017,YangyangAugLag17}.
Proximal quadratic penalty (PQP)
 methods were first studied in \cite{WJRproxmet1} and further developed in \cite{WJRComputQPAIPP, MinMax-RenWilliam,PPmetNonconvex2019}.  

Classical proximal AL (PAL) methods for solving \eqref{optl0} under the assumption that $f$ is convex, $\theta=0$, and $\chi_k=1$ for every $k$ have first been studied in \cite{rockafellar1976augmented}. Recently, papers \cite{HongPertAugLag,RJWIPAAL2020} studied PAL methods under the assumption that $f$ is (possibly) nonconvex, $\theta \in (0,1]$, and
$\chi_k=1$ for every $k$. However, as $\theta$ approaches zero, the
prox stepsize $\lam$ of these methods converge to zero
which causes the following issues:
(i) their derived complexity bounds 
diverge to infinity (see
the second column in Table~\ref{tab:summary_tbl} below), which invalidates their analyses for
the case of $\theta=0$; and (ii) deteriorating computational performance. 

Papers \cite{ImprovedShrinkingALM20, inexactAugLag19} 
study non-proximal AL methods under the assumption that $f$ is nonconvex, $\theta=0$, and $\chi_k=O(c_k^{-1})$ for every $k$. It is worth mentioning that both \cite{ImprovedShrinkingALM20, inexactAugLag19} make a strong assumption about the generated iterates (see condition ${\cal F}$ in Table~\ref{tab:conditions}), which have only been shown to hold when $h$ is the indicator of a polyhedron or a ball. Moreover, \cite{inexactAugLag19} considers the more general version of \eqref{optl0} where the constraints are (possibly) nonconvex.

We now describe other papers that are tangentially related to ours. 
Papers \cite{ErrorBoundJzhang-ZQLuo2020,ADMMJzhang-ZQLuo2020}  present a  primal-dual first-order algorithm under the assumption that $h$ is the indicator function
of a box (in \cite{ADMMJzhang-ZQLuo2020}) or a polyhedron (in \cite{ErrorBoundJzhang-ZQLuo2020}). 
Paper \cite{SZhang-Pen-admm} considers a penalty-ADMM method that solves an equivalent reformulation of \eqref{optl0}. 
Paper \cite{HybridPenaltyAugLag19} presents an inexact proximal point method applied to the function defined as $\phi(z)$ if $z$ is feasible and $+\infty$ otherwise. 
It can be viewed as an extension to the nonconvex setting
of the proximal point method (PPM) applied to
\eqref{optl0} (see, for example, \cite{rockafellar1976augmented} for the analysis of inexact versions of  PPMs for solving \eqref{optl0} in the convex setting). Paper \cite{Lan-ConstrainedStocasticProxMetNonconvex2019} considers a  primal-dual proximal point  scheme for computing approximate stationary solution to a  constrained  NCO problem and  analyzes  its iteration-complexity under different assumptions.  

Before closing this review, we present the assumptions of some of the above methods in Table~\ref{tab:conditions} and give a summary of these methods in
Table~\ref{tab:summary_tbl}, which compares the best iteration complexities, necessary
conditions, and various parameter ranges. 
\begin{table}[!tbh]
\begin{centering}
\begin{tabular}{c>{\raggedright}m{0.65\columnwidth}}
\toprule 
{\scriptsize{}${\cal B}$} & {\scriptsize{}Either (i) the quantity $\sup_{x\in\dom h}|\phi(x)|$
is finite, (ii) $\dom h$ is bounded, and/or (iii) the feasible set
is bounded.}\tabularnewline
\midrule 
{\scriptsize{}${\cal A}$} & {\scriptsize{}If the constraints have an affine component of the form
$Ax=b$ then $A$ has full row rank.}\tabularnewline
\midrule 
{\scriptsize{}${\cal F}$} & {\scriptsize{}There exists some $\nu>0$ such that $\nu \|Ax_{k}-b\| \leq {\rm dist}(0,A^{*}(Ax_{k}$ $-b)+c_{k}^{-1}\partial h(x_{k}))$
for generated iterates $\{x_{k}\}_{k\geq1}$ and
$\{c_{k}\}_{k\geq1}$.}\tabularnewline
\midrule 
{\scriptsize{}${\cal N}$} & {\scriptsize{}The function $h$ restricted to its domain is $r$-Lipschitz
continuous.}\tabularnewline
\midrule 
{\scriptsize{}${\cal SP}$} & {\scriptsize{}There exists $\bar{x}\in{\rm int}(\dom h)$
such that $Ax=b$.}\tabularnewline
\bottomrule
\end{tabular}
\par\end{centering}
\caption{Abbreviations for common boundedness and regularity conditions. See
Lemma~\ref{Prop:eps-subd-bounded-cone} for the result that ${\cal N}$ is equivalent to requiring that,
for every $x\in\protect\dom h$, there exists $r>0$ such that that
$\protect\partial h(x)\subseteq{\cal N}_{\protect\dom h}(x)+{\cal B}_{r}$
where ${\cal B}_{r}=\{x:\|x\|\protect\leq r\}$. \label{tab:conditions}}
\end{table}
\begin{table}[!tbh]
\begin{centering}
\begin{tabular}{ccccccc}
\toprule 
\textbf{\scriptsize{}Name} & \textbf{\scriptsize{}Best Complexity} & {\scriptsize{}$\lam_k$} & \textbf{\scriptsize{}$\theta$} & \textbf{\scriptsize{}$\chi_{k}$} & \textbf{\scriptsize{}Conditions} & \tabularnewline
\midrule 
{\scriptsize{}QP-AIPP \cite{WJRproxmet1}} & {\scriptsize{}${\cal O}(\varepsilon^{-3})$} & {\scriptsize{}$\Theta(m_{f}^{-1})$} & {\scriptsize{}-} & {\scriptsize{}-} & {\scriptsize{}None} \tabularnewline
{\scriptsize{}R-QP-AIPP \cite{WJRComputQPAIPP}} & {\scriptsize{}$\tilde{{\cal O}}(\varepsilon^{-3})$} & {\scriptsize{}$(0,\infty)$} & {\scriptsize{}-} & {\scriptsize{}-} & {\scriptsize{}${\cal B}$}\tabularnewline
{\scriptsize{}iPPP \cite{PPmetNonconvex2019}} & {\scriptsize{}$\tilde{{\cal O}}(\varepsilon^{-5/2})$} & {\scriptsize{}${\cal O}(m_{f}^{-1})$} & {\scriptsize{}-} & {\scriptsize{}-} & {\scriptsize{}${\cal B},{\cal N},{\cal SP}$}\tabularnewline
\midrule 
{\scriptsize{}iALM (2019) \cite{inexactAugLag19}} & {\scriptsize{}$\tilde{{\cal O}}(\varepsilon^{-3})$} & {\scriptsize{}-} & {\scriptsize{}$0$} & {\scriptsize{}${\cal O}(c_k^{-1})$} & {\scriptsize{}${\cal B},{\cal F}$}\tabularnewline
{\scriptsize{}iALM (2020) \cite{ImprovedShrinkingALM20}} & {\scriptsize{}$\tilde{{\cal O}}(\varepsilon^{-5/2})$} & {\scriptsize{}-} & {\scriptsize{}$0$} & {\scriptsize{}${\cal O}(c_k^{-1})$} & {\scriptsize{}${\cal B},{\cal F}$}\tabularnewline
{\scriptsize{}PProx-PDA\tablefootnote{This method generates prox subproblems of the form $\argmin_{x\in X}\{\lam h(x) + c\|Ax-b\|^2 / 2 + \|x-x_0\|^2 / 2 \}$ and the analysis of \cite{HongPertAugLag}
makes the strong assumption that they can be solved exactly for any $x_0$, $c$, and $\lam$.} \cite{HongPertAugLag}} & {\scriptsize{}${\cal O}(\theta^{-2}\varepsilon^{-4})$} & {\scriptsize{}${\cal O}(\theta L_{f}^{-1})$} & {\scriptsize{}$(0,1)$} & {\scriptsize{}$1$} & {\scriptsize{}${\cal B}$, ${\cal A}$}\tabularnewline
{\scriptsize{}$\theta$-IPAAL\tablefootnote{It is also shown that conditions $\cal N$ and $\cal SP$ can be removed to yield a  complexity of $\tilde{{\cal O}}(\theta^{-7/2}\varepsilon^{-3})$.} \cite{RJWIPAAL2020}} & {\scriptsize{}$\tilde{{\cal O}}(\theta^{-15/4}\varepsilon^{-5/2})$} & {\scriptsize{}$\Theta(\theta m_{f}^{-1})$} & {\scriptsize{}$(0,1)$} & {\scriptsize{}$1$} & {\scriptsize{}${\cal N},{\cal SP}$}\tabularnewline
\textbf{\scriptsize{}IAIPAL
}
& {\scriptsize{}$\tilde{{\cal O}}(\varepsilon^{-5/2})$} & {\scriptsize{}$\Theta(m_{f}^{-1})$} & \textbf{\scriptsize{}$0$} & \textbf{\scriptsize{}$\{0,1\}$\tablefootnote{Specifically, $\chi_k=1$ if $k\equiv 1\bmod k_\circ$ and $\chi_k=0$ otherwise.}} & {\scriptsize{}${\cal B},{\cal N},{\cal SP}$}\tabularnewline
\bottomrule
\end{tabular}
\par\end{centering}
\caption{Comparison of relevant penalty and AL-based methods with IAIPAL. For
simplicity, we let $\varepsilon=\min\{\hat{\rho},\hat{\eta}\}$, and
let $\tilde{{\cal O}}(\cdot)$ be the same as ${\cal O}(\cdot)$ with
all logarithmic dependencies on $\varepsilon$ removed. }
\label{tab:summary_tbl}
\end{table}

\noindent\textit{Contributions.} 
Under the assumption that the domain of $h$ is bounded,  has nonempty interior, and
	\eqref{optl0} has a point $\bar z \in \inte(\dom h)$ such that $A\bar z=b$,
	it is shown that if $\alpha = \Theta(1)$ then the total
	ACG iteration complexity
	of IAIPAL, up to logarithmic terms,  is 
	\begin{equation}
	\label{eq:contrib_compl}
    {\cal O}\left(\frac{1}{\hat\rho^{5/2}}+\frac{1}{ \sqrt{\hat\eta}\hat\rho^{2}}\right).
	\end{equation} 
    which is equal, up to logarithmic terms, to the ones obtained for the methods in \cite{ImprovedShrinkingALM20, PPmetNonconvex2019, RJWIPAAL2020} (see Table~\ref{tab:summary_tbl}). On the other hand, if $\alpha=1/c$, i.e., a full multiplier update is performed at every step of IAIPAL in view of \eqref{def:pk0}, then it is shown that 
 	the above complexity becomes ${\cal O}(\hat\rho^{-3}+\hat\rho^{-2}\eta^{-1/2})$.
	Since each ACG iteration of IAIPAL requires ${\cal O}(1)$ resolvent evaluations of $h$ and/or gradient evaluations of $f$, the above complexities
	also bound the number of $h$-resolvent computations
	(i.e., evaluations of $(I+ \eta \partial h)^{-1}$
for $\eta>0$)
and gradient evaluations of $f$ performed by
	IAIPAL. It is also worth mentioning
	that all of the above results hold
    without assuming that the
	initial point $z_0 \in \dom h$ is feasible, i.e., $z_0$ also
	satisfies $Az_0=b$.
	
	We now emphasize four important theoretical aspects of this paper.
	First, it establishes (for the \textit{first} time) the
    iteration-complexity of a PAL method (specifically IAIPAL with $\alpha=1/c$) for solving \eqref{optl0} which makes a full multiplier update (i.e., $(\theta,\chi_k)=(0,1)$ for every $k$) after solving each prox subproblem, does not assume boundedness of the multiplier sequence $\{p_k\}$, and contains a novel rule for updating the penalty parameter.
	Second, the proof that the sequence of Lagrange multipliers is bounded does not use
potential function arguments (e.g. as in \cite{HongPertAugLag,ProxAugLag_Ming,RJWIPAAL2020}), restrict the size of $\chi_k$ in \eqref{eq:gen_dual_update2} (e.g. as in \cite{ImprovedShrinkingALM20,inexactAugLag19}), and/or limit the number of multiplier updates (e.g. as in \cite{ImprovedShrinkingALM20,inexactAugLag19}). 
Third, in contrast to
the PAL methods of \cite{HongPertAugLag,RJWIPAAL2020},
whose iteration-complexities and stepsizes tend to $\infty$ and $0$, respectively, as $\theta$ tends to $0$ (see the second column in Table~\ref{tab:summary_tbl}),
the complexity and stepsize of IAIPAL do not depend on $\theta$. Fourth, in contrast to the non-proximal AL methods of \cite{ImprovedShrinkingALM20, inexactAugLag19}, which make  strong assumptions on the generated iterates (see condition $\cal F$ in Table~\ref{tab:conditions}), the convergence of IAIPAL only assumes a mild Slater-like condition and Lipschitz continuity of $h$ on its domain.

	It is also worth mentioning that the numerical experiments of Section~\ref{sec:numerical}, and the conclusions thereof, show that IAIPAL with $\alpha=\Theta(1)$ and $\alpha=\Theta(1/c)$ substantially outperforms other algorithms in the literature for solving \eqref{optl0}
	(or special cases of it) with equal (e.g., \cite{WJRproxmet1,WJRComputQPAIPP,PPmetNonconvex2019,RJWIPAAL2020}) or better (e.g., \cite{ErrorBoundJzhang-ZQLuo2020,ADMMJzhang-ZQLuo2020}) iteration complexities.
\\

\noindent\textit{Organization of the paper}.  Subsection~\ref{sec:bas}  provides some basic definitions and notation. Section~\ref{sec:IAIPALmethods} contains three subsections. The first one presents our main problem of interest and the assumptions made on it. The second one  states S-IAIPAL and its main iteration-complexity result. Subsection~\ref{Sec:dynamic-IPAAL} states IAIPAL and establishes its iteration-complexity bound. Section~\ref{sec:technical} is devoted to the proof of the iteration-complexity result of S-IAIPAL and some related technical results. Section~\ref{sec:numerical}
 presents some numerical experiments comparing IAIPAL with other benchmarks algorithms for solving \eqref{optl0}. Section~\ref{sec:conclusion} contains some concluding remarks. Finally, an appendix section is considered and it  is divided into three subsections. Subsection~\ref{sec-nesterov}  reviews an ACG method used to solve the S-IAIPAL subproblems. The second subsection contains a basic result of convex analysis, and the last subsection presents a basic lemma associated with a refinement procedure considered in S-IAIPAL.




	\subsection{Notation and basic definitions}
	\label{sec:bas}
	This subsection presents   notation and basic definitions used in this paper.
	
	Let $\mathbb{N}$ denote the set of positive integers. Let $\Re_{+}$  and $\Re_{++}$ denote the
	set of nonnegative and positive
	real numbers, respectively, and
	let $\Re^n$ denote the $n$-dimensional Hilbert space with inner product and associated norm denoted by $\inner{\cdot}{\cdot}$ and $\|\cdot\|$, respectively. We use $\Re^{l\times n}$ to denote the set
	of all $l\times n$ matrices and ${\mathbb S}_n^+$ to denote the set of positive semidefinite matrices in $\r^{n\times n}$.  
	The smallest positive singular value
	of a  nonzero linear operator $Q:\Re^n\to \Re^l$
	is denoted by $\sigma^+_Q$.
For a given closed convex set $X \subset \Re^n$, its boundary is denoted by $\partial X$ and the distance of a point $x \in \Re^n$ to $X$ is denoted by ${\rm dist}_X(x)$.
 For any $t>0$, we let $\log_1^+(t):=\max\{\log t, 1\}$ and  $\bar{B}(0,t):=\{z\in \Re^n:\|z\|\leq t\}$.
	
	The domain of a function $h :\Re^n\to (-\infty,\infty]$ is the set   $\dom h := \{x\in \Re^n : h(x) < +\infty\}$.
	Moreover, $h$ is said to be proper if
	$\dom h \ne \emptyset$. The set of all lower semi-continuous proper convex functions defined in $\Re^n$ is denoted by $\bConv{n}$. The $\varepsilon$-subdifferential of a proper function $h :\Re^n\to (-\infty,\infty]$ is defined by 
	\begin{equation}\label{def:epsSubdiff}
	\partial_\varepsilon h(z):=\{u\in \Re^n: h(z')\geq h(z)+\inner{u}{z'-z}-\varepsilon, \quad \forall z' \in \Re^n\}
	\end{equation}
	for every $z\in \Re^n$.	The classical subdifferential, denoted by $\partial h(\cdot)$,  corresponds to $\partial_0 h(\cdot)$.  Recall that, for a given $\varepsilon\geq 0$, the $\varepsilon$-normal cone of a closed convex set $C$ at $z\in C$, denoted by  $N^{\varepsilon}_C(z)$, is defined as  
	$N^{\varepsilon}_C(z):=\{\xi \in \Re^n: \inner{\xi}{u-z}\leq \varepsilon,\quad \forall u\in C\}.$
	If $\psi$ is a real-valued function which
	is differentiable at $\bar z \in \Re^n$, then its affine   approximation $\ell_\psi(\cdot,\bar z)$ at $\bar z$ is given by
	\begin{equation}\label{eq:defell}
	\ell_\psi(z;\bar z) :=  \psi(\bar z) + \inner{\nabla \psi(\bar z)}{z-\bar z} \quad \forall  z \in \Re^n.
	\end{equation}

	\section{The  IAIPAL method}\label{sec:IAIPALmethods}
This section is divided into three subsections. The first one discusses the problem of interest and describes the main
assumptions made on it. 
Subsection~\ref{Sec:static-IPAAL} presents S-IAIPAL and its main iteration-complexity result. Subsection~\ref{Sec:dynamic-IPAAL} presents IAIPAL and its overall ACG  iteration-complexity result. 
	\subsection{Problem of interest, assumptions and IAIPAL outline}\label{subsec:assump-Motivation}
	This subsection  describes
	the problem of interest, the assumptions made on it,
	and the type of  approximate stationary  solution we are interested
	in computing for it.

	The main problem of interest in this paper is \eqref{optl0} 
	where  $f, h:  \Re^{n} \to (-\infty,\infty]$,
	${A} : \Re^n \to \Re^l$ and $b \in \Re^l$ satisfy the following assumptions:
	\begin{itemize}
		\item[{\bf(B1)}] $A$ is a nonzero linear operator;
		\item[{\bf(B2)}] $h\in \bConv{n}$ is $L_h$-Lipschitz continuous on ${\mathcal H}:=\dom h$;
		\item[{\bf(B3)}] the diameter $D:=\sup \{ \|z-z'\| : z, z' \in {\mathcal H} \}$ of ${\mathcal H}$ is finite and there exists
		$\nabla_f \ge 0$ such that
		$\|\nabla f(z) \| \le \nabla_f$ for every $z \in {\mathcal H}$;
		\item[{\bf(B4)}] there exists $\bar z \in \inte({\mathcal H})$ such that $A\bar z=b$;
		\item[{\bf(B5)}] $f$ is nonconvex and  differentiable on $\rn$, and there exist $L_f\geq m_f > 0$ such that, for all $z,z' \in \rn$,
			\begin{align}\label{gradLips}
		\|\nabla f(\z') -  \nabla f(\z)\| &\le L_f \|z'-z\|,\\
		f(z') -\ell_f(z';z) &\ge - \frac{m_f}2 \|z'-z \|^2.\label{lowerCurvature-m}
		\end{align}
		\end{itemize} 
	
	
		Some comments about assumptions
		{\bf (B1)}--{\bf (B5)} are in order. First,
		it is shown in Lemma~\ref{Prop:eps-subd-bounded-cone} that
$\partial_{\varepsilon}h(z) \subset \bar{B}(0,L_h)+N^{\varepsilon}_{{\mathcal H}}(z)$ for every $z \in {\mathcal H}$. This inclusion will be used to bound the sequence  of Lagrangian multipliers  generated by the IAIPAL method.
		 Second,  it is well known that \eqref{gradLips} implies that 
	$ |f(z') -\ell_f(z';z) | \le  L_f\|z'-z \|^2/2$ for every $z,z' \in \rn$, and hence that \eqref{lowerCurvature-m} holds with $m_f=L_f$. However, better
	iteration-complexity bounds
	can be derived when a scalar
	$m_f< L_f$ 
	satisfying \eqref{lowerCurvature-m}
	is available.
	  Third, \eqref{lowerCurvature-m} implies that the function $f(\cdot)+ m_f \|\cdot\|^2/2$ is convex on $\rn$. Moreover, since $f$ is nonconvex on $\rn$ in view of {\bf(B5)}, the smallest $m_f$ satisfying  \eqref{lowerCurvature-m} is positive. 
Fourth, 
any function $f$ of the form $h=\tilde h + \delta_Z$
where $\tilde h$ is a finite everywhere Lipschitz continuous
convex function
and $Z$ is a compact convex set clearly satisfies condition (B2).
Finally, the existence of a scalar $\nabla_f$  as  in  {\bf (B3)} is
actually not an extra assumption since,
using
\eqref{gradLips} and 
the boundedness of ${\mathcal H}$ in  {\bf (B3)},
it can be easily seen
that for any $y \in {\mathcal H}$,
the scalar
$\nabla_f = \nabla_{f,y} := \|\nabla f(y)\| +
L_f D$ majorizes $\|\nabla f(z)\|$ for any
$z \in {\mathcal H}$.

	It is well known that, under some mild conditions, if $\bar{z}$ is a local minimum of \eqref{optl0}, then there exists
	$\bar{\lm} \in \Re^l$ such that $ (\bar{z}, \bar{\lm})$ is a stationary solution of \eqref{optl0}, i.e., 
	\begin{equation}\label{estacionarypoint}
	0\in  \nabla f(\bar{z})+\partial h(\bar{z})+A^*\bar{\lm},
	\quad A\bar{z}-b=0.
	\end{equation}
	The main complexity results of this paper are stated in terms of the following notion of approximate  stationary solution
	which is a natural relaxation of \eqref{estacionarypoint}.

	\begin{definition}\label{def:stationarypoint}
		Given a tolerance pair  $(\hat\rho,\hat\eta)\in \Re_{++}\times\Re_{++}  $, a triple $(\hat z, \hat \lm,\hat w)\in {\mathcal H}\times\Re^l\times\Re^n$ is said to be a $(\hat\rho,\hat\eta)$-approximate stationary  solution of \eqref{optl0} if it satisfies \eqref{eq:approx_stationary1'}.
		\end{definition}

	\subsection{The S-IAIPAL method}\label{Sec:static-IPAAL}
	
This subsection describes S-IAIPAL,
which essentially corresponds to
to some group of all consecutive iterations
of the general IAIPAL method outlined in Section~\ref{sec:int} (see the paragraph containing \eqref{exactzk0}-\eqref{def:pk0}) 
for which the penalty parameter $c$ stays
constant.


        Recall from the outline given
      in the Introduction that S-IAIPAL generates
        a sequence $\{(z_k,p_k)\}$ according to \eqref{exactzk0} and \eqref{def:pk0} where $\lambda=1/(2m_f)$.
        The formal description of S-IAIPAL below requires that,
        for a pre-specified scalar $\tilde \sigma>0$, the
        approximate solution $z_k$ of subproblem \eqref{exactzk0},
        together with some residual pair $(v_k,\varepsilon_k)\in \Re^n\times\Re_{++}$, satisfy
        \begin{equation}\label{GIPPinc_ineq}
		v_k \in \partial_{\varepsilon_k} \left ( \lam {\cal L}_c(\cdot,\lm_{k-1}) + \frac12 \|\cdot-\z_{k-1}\|^2 \right) (\z_k),
		\qquad \| v_k\|^2 + 2 \varepsilon_k \le \tilde \sigma^2 \|r_k\|^2
		\end{equation}
		where
		\begin{equation}\label{def:rk}
		r_k:= z_{k-1}-z_k+v_k.
		\end{equation}
		
		We now make some remarks about the above notion of approximate solution for \eqref{exactzk0}.
		First, even though $\tilde \sigma$ is assumed to positive, it is worth noting that if $\tilde \sigma$ were equal to zero, then \eqref{GIPPinc_ineq} would immediately imply that $z_k$ is the exact solution of \eqref{exactzk0}.
		Hence, the aggregated error $\| v_k\|^2 + 2 \varepsilon_k$ of the residual pair $(v_k,\varepsilon_k)$ can be thought as an inexactness measure of the approximate solution $z_k$,
		and the inequality in \eqref{GIPPinc_ineq} is  a relative error condition on it.
		Second, as will be seen in Proposition~\ref{Prop:SIAIPAL-ACG} below,
		a triple $(z_k,v_k,\varepsilon_k)$ satisfying \eqref{GIPPinc_ineq}
		can be found by suitably applying the ACG method
		described in Subsection~\ref{sec-nesterov} to
		subproblem \eqref{exactzk0}.
		

We now formally describe S-IAIPAL.

    \vgap
	\hrule
	\vspace*{0.5em}
	\noindent\quad{\bf S-IAIPAL}
	\vspace*{0.5em}
	\hrule
	\vspace*{0.5em}
	\begin{itemize}
		\item[(0)] Let  scalars  $\nu>0$
		and $\sigma \in (0, 1/\sqrt{2}]$,
		initial point $z_0 \in {\mathcal H}$,
		tolerance pair $(\hat \rho,\hat \eta) \in \Re_{++}\times\Re_{++}$, penalty parameter $c \geq 0$, and $\alpha > 0$
		be given; set $k=1$, $p_0=0$, and
		\begin{equation}\label{definition of sigma}
		\begin{gathered}
		C_1=\frac{2\left(1+ 2\nu\right)^2}{1-\sigma^2}, \quad \lambda=\frac{1}{2m_f},\quad\sigma_c= \min\left\{\frac{\nu}{\sqrt{\lambda L_c+1}},\sigma \right\}, \\
	L_c=L_f+c\|A\|^2;
		\end{gathered}
		\end{equation}
			\item [(1)] 
		use the ACG method described in  Subsection~\ref{sec-nesterov} with inputs
		\begin{align}
		x_0&=z_{k-1}, \quad \tilde \sigma=\sigma_c, \quad
		({\widetilde \mu},{\widetilde M})= (1/2,\lam L_c+1), \label{eq:Ms-mu}\\
		(\psi^{(s)},\psi^{(n)})&= 
	\left( \lam [\mathcal{L}_c(\cdot,p_{k-1})-h] +\frac{1}{2}\|\cdot-z_{k-1}\|^2\,,\, 
	\lam h \right)\label{eq:psiS-psimu}
	\end{align}
		to obtain a triple $(z_k,v_k,\varepsilon_k)$ satisfying \eqref{GIPPinc_ineq} with $\tilde \sigma = \sigma_c$, and set
		\begin{equation}
		\label{def:pk}
		q_k = p_{k-1} + c(Az_k - b), \quad \lm_k = \begin{cases}
                q_k, & k\equiv 1\bmod \lceil\alpha c\rceil,\\
                p_{k-1}, & \text{otherwise};
        \end{cases}
		\end{equation}
		\item[(2)] compute $(\hat z_k,\hat \lm_k, \hat w_k)$ as 
		\begin{gather}
        \hat  z_k = \argmin_u \left\{\langle \lambda \left[\nabla f(z_k)+A^*q_k\right]-r_k, u\rangle+ \lambda h(u)+ \frac{\lambda L_c+1}{2}\|u-z_{k}\|^2\right \}, \nonumber \\
		 \hat \lm_k =p_{k-1}+c(A\hat z_k -b ),\label{def:pkhat} \\
		\hat w_k =  \frac{1}{\lambda}r_k + \left(L_c + \frac{1}{\lam} + cA^*A\right)(\hat z_k-z_k) +  \nabla f(\hat z_k)-\nabla f(z_k) ,\label{def:wkhat}
		\end{gather}
		where $r_k$ is as in \eqref{def:rk};
		if $\|\hat w_k\| \le \hat \rho$ and
		$\|A\hat z_k - b\| \le \hat \eta$,
		then {\bf stop with success}  and output
		$(\hat z, \hat \lm,\hat w)=(\hat z_k, \hat \lm_k,\hat w_k)$;
		\item[(3)] if $k\geq 2$ and
		\begin{equation}\label{definition:Deltak}
		\Delta_k = \frac{{\cal L}_\pen(z_1,\lm_{1}) - {\cal L}_\pen(z_k,\lm_{k})}{k-1}\leq \frac{\lam\hat \rho^2}{2C_1},
		\end{equation}
		then {\bf stop and  declare
		$c$ small};
	\item[(4)] 
	set
	$k \leftarrow k+1$, and go to step 1.
	\end{itemize}
	\vspace*{0.5em}
	\hrule
	\vgap


	We now make some trivial remarks about S-IAIPAL.
	First,  it performs  two types of iterations, namely,
	the outer ones indexed by $k$ and the ACG (or inner) ones performed
	during its calls to ACG in step~1. Second,
  the scalar $\lambda$ defined in step~0 ensures that the prox  augmented Lagrangian subproblem  \eqref{exactzk0} is strongly convex. Third, the scalars ${\widetilde M}$ and ${\widetilde \mu}$ in step~1 are the Lipschitz constant and the strong convexity parameter of $\nabla \psi_s$ and $\psi_n$, respectively. 
  Fourth, the update formula \eqref{def:pk} for the multiplier $p_k$ is the classical one where a full step is performed, i.e., no shrinking  factor multiplying  the term $c(Az_k-b)$ is included on it. Fifth, it follows immediately from \eqref{def:pk} and \eqref{def:pkhat} that
	\begin{equation}
	    \label{Relation:pkhat-pk}
	    \hat p_k-p_k=cA(\hat z_k-z_k) \quad \forall k\equiv 1\bmod \lceil \alpha c\rceil.
	\end{equation}
  
  We next make some comments about the logical structure of S-IAIPAL.
  First,  it is shown in
  Proposition~\ref{prop:approxsol} that every triple
  $(\hat z, \hat p,\hat w)=(\hat z_k, \hat p_k,\hat w_k)$ computed in step~2 satisfies the inclusion in \eqref{eq:approx_stationary1'}, and hence is a 
  $(\hat \rho,\hat\eta)$-approximate stationary solution of \eqref{optl0}
  (see Definition~\ref{def:stationarypoint})
  whenever S-IAIPAL  stops successfully (see the condition for that to happen at the end of step 2).
  Second, in contrast to
  the $k$-th generated triple
  $(\hat z_k,\hat p_k,\hat w_k)$, which is only used in step~2
  to test for possible termination,
  the $k$-th generated quadruple $(z_k,p_k,v_k,\varepsilon_k)$ found in step~1 is not only used to compute the above triple but also to perform the next iteration.
  Third, Theorem~\ref{theor:StaticIPAAL}(d) below shows that if the penalty parameter $c$ is sufficiently large at some iteration, then S-IAIPAL must successfully stop
  in its step~2.
  Finally, after the second iteration (and including it) of S-IAIPAL, inequality \eqref{definition:Deltak}  is used to  detect  whether the penalty parameter $c$ is small, in which case S-IAIPAL stops in its step~3 with the declaration that $c$ is small.
  IAIPAL, which is discussed in the next subsection, then uses this information to increase $c$ and restart
	S-IAIPAL with the new value of $c$ and
	with the initial point
	$z_0$ either set to be
	the same as in the previous
	S-IAIPAL call,
	i.e., $z_0$ is kept constant
	(cold S-IAIPAL restart), or
	set to be equal to $z_k$, 
	where $z_k$ is the iterate
	computed in step~1 of S-IAIPAL just before it declares $c$ small (warm S-IAIPAL restart).

 

The following result describes an upper bound on
the number of iterations performed during each call to ACG in step 1 of S-IAIPAL.
\begin{proposition}\label{Prop:SIAIPAL-ACG}
Each call to
the ACG method in step 1 of S-IAIPAL performs at most   
			\begin{equation}\label{eq:ACGbound-porouter}
			\left\lceil 5\left(\sqrt{\frac{2L_f}{m_f}}+\sqrt{\frac{c\|A\|^2}{m_f}}\right)\log^+_1 \mathcal{M}(c)
			\right\rceil
			\end{equation}
ACG iterations, where 
$\mathcal{M}(c)$ is given by
\begin{equation}\label{def:Thetac}
\mathcal{M}(c) = 2 \left[ \frac{3L_f}{m_f}+\frac{c\|A\|^2}{m_f} \right] \max\{ \nu^{-1},\sigma^{-1} \}.
\end{equation}
\end{proposition}
\begin{proof}
First note that the respective definitions of ($\lambda$, $\sigma_c$, $L_c$), $(\tilde \sigma,{\widetilde \mu},{\widetilde M})$, and ${\cal A}_{\widetilde\mu,\widetilde\sigma}$ in \eqref{definition of sigma}, \eqref{eq:Ms-mu}, and Proposition~\ref{prop:nest_complex}, together with the bounds $\sigma_c<1$ and $L_f/m_f \geq 1$ from the definition of $\sigma_c$ and {\bf (B5)}, imply that  
\begin{align*}
{\cal A}_{\widetilde\mu,\sigma_c}&=\frac{4(1+\sigma_c)^{2}}{\sigma_c^2}\leq\frac{16}{ \sigma_c^2}\leq 8\left(\frac{3L_f}{m_f}+\frac{c\|A\|^2}{m_f}\right){\max\{\nu^{-2},\sigma^{-2}\}}, \\
\widetilde M-\widetilde\mu&=\lambda L_c+\frac{1}{2}=\frac{L_f+c\|A\|^2}{2m_f}+\frac{1}{2}\leq\frac{1}{2}\left(\frac{2L_f}{m_f}+\frac{c\|A\|^2}{m_f}\right). 
\end{align*}
Hence,  \eqref{eq:ACGbound-porouter}   follows from  Proposition~\ref{prop:nest_complex},  the above inequalities, the definition of  $\mathcal{M}(c)$ in \eqref{def:Thetac}, and the fact that $\log^+_1(\cdot)\geq 1$.  

\end{proof}

The following quantities and constants will be used
in the  statement and
proof of the main result of this subsection
(Theorem~\ref{theor:StaticIPAAL} below).
\begin{align}
&C_2:=\frac{\sigma^2}{(1-\sigma)^2},\quad C_3:=\frac{1+\nu}{1-\sigma},\label{definition:C3}    \\
&\phi_*:=\inf_{z\in \Re^n} \phi(z),\quad \Delta\phi^*=\phi^*-\phi_*, \quad 
	\bar{d}:=\mbox{\rm dist}_{\partial {\mathcal H}} (\bar z),  \label{def:zeta} \\
	&\theta_A:=\frac{\|A\|}{\sigma^+_A}, \quad \theta_D = \frac{D}{\bar d},\label{eq:condNumb}\\
	 &\kappa_0:=2(L_h+\nabla_f) +(C_2+4C_3)m_fD,\label{def:kappa00} \\
	&\kappa_1:=2\sqrt{2C_1}\theta_A\theta_D\kappa_0, \qquad \kappa_2:=\left(\frac{5\|A\| \theta_A\theta_D \kappa_0}{2m_f} \right)^{1/2},\label{def:kappa1-kappa2}
\end{align}
where $\phi^*$ is as in \eqref{optl0}, $C_1$ is as in \eqref{definition of sigma},  $D$ and $\nabla_f$ are as in {\bf (B3)}, and  $\bar{z}$  is as in  {\bf (B4)}.
Note that $C_1$, $C_2$ and $C_3$ are constants depending
only on the input parameters $\nu$ and/or $\sigma$
of S-IAIPAL. Moreover, the constants $\kappa_0$,
$\kappa_1$ and $\kappa_2$ depend not only on the constants
$C_1$, $C_2$ and $C_3$, but also on the constants $D$, $\|A\|$, $L_h$, $m_f$, $\nabla_f$,
and the ones defined in \eqref{def:zeta} and \eqref{eq:condNumb},
which are all associated
with the instance of  \eqref{optl0} under consideration.
Constants $\kappa_1$ and $\kappa_2$ are in turn used
to describe a threshold value $\bar c$ (see \eqref{def:cbar} below)
such that if
$c \ge \bar c$ then
S-IAIPAL is guaranteed to terminate with a $(\hat\rho,\hat\eta)$-approximate stationary  solution
of \eqref{optl0} (see statement (d) below).




 Next we state the main result about S-IAIPAL, whose proof is given  at the end of Section~\ref{sec:technical}.

\begin{theorem}\label{theor:StaticIPAAL} Assume that $c\geq m_f / \|A\|^2$ and that conditions  {\bf (B1)}--{\bf (B5)} hold.
Then, the following statements about S-IAIPAL hold:

\begin{itemize}
    \item[a)]  every iterate
$(\hat z_k,\hat p_k,\hat w_k)$  with $k\geq 1$ satisfies
$
\hat w_k\in \nabla f(\hat z_k)+\partial h(\hat z_k)+ A^* \hat p_k;
$
\item[b)]
	the number of outer iterations is bounded by
\begin{equation}
    \label{eq:outernumberiter}
    T_0(\hat \rho)
    := \left \lceil 1 + \frac{12 C_1 m_f\left(\Delta\phi^*+2m_fD^2\right)+ \kappa_1^2/4}{\hat\rho^2} \right \rceil,
	\end{equation}
where $C_1$, $\Delta\phi^*$, $\theta_D$, $\kappa_1$, $D$, and $m_f$ are as in step~0 of S-IAIPAL, \eqref{def:zeta}, \eqref{eq:condNumb}, \eqref{def:kappa1-kappa2}, \textbf{(B3)}, and \textbf{(B5)}, respectively; hence, the  total number of ACG iterations is bounded by
\begin{equation}\label{TotalACG_perouter}
T_{ACG}(c,\hat\rho):=\left\lceil 5 \left ( \sqrt{\frac{2L_f}{m_f}}   +  \sqrt{\frac{c\|A\|^2}{m_f} }\right)\log^+_1\mathcal{M}(c) \right\rceil
T_0(\hat \rho),
\end{equation}
where $L_f$ and $\mathcal{M}(c)$ are as in \textbf{(B5)} and \eqref{def:Thetac}, respectively.
\end{itemize}
Moreover, if the penalty parameter $c$ satisfies
\begin{equation}
\label{def:cbar-eta-rho}
c \geq   \frac{\kappa_2^2 m_f}{\hat{\eta} \|A\|^2}, \quad c\lceil{c\alpha} \rceil
\geq \frac{4 m_f \kappa_1^2}{\hat{\rho}^2\|A\|^2}
\end{equation}
then the following statements also hold:
\begin{itemize}
\item [c)] every  iterate
$(\hat z_k,\hat p_k,\hat w_k)$  with $k\geq 1$ satisfies
$
 \|A\hat z_k - b \| \le \hat \eta; 
$
\item[d)] S-IAIPAL stops successfully in step~2 with
a $(\hat\rho,\hat\eta)$-approximate stationary solution $(\hat z,\hat p,\hat w)$
of \eqref{optl0}.
\end{itemize}	  
\end{theorem}




We now make some  remarks about Theorem~\ref{theor:StaticIPAAL} under the condition that the penalty parameter $c$ satisfies $\hat{c}_\alpha \le c= {\mathcal O}(\hat{c}_\alpha)$ where 
\begin{equation}\label{def:cbar}
\hat{c}_\alpha = \hat{c}_\alpha(\hat{\rho},\hat{\eta}) := \min \left\{c: c \mbox{ satisfies \eqref{def:cbar-eta-rho}}\right\}.
\end{equation}
First,  it follows  from parts b) and d) that
S-IAIPAL obtains a $(\hat\rho,\hat\eta)$-approximate stationary  solution of \eqref{optl0} in $\mathcal{O}\left(T_{ACG}(\hat{c}_\alpha,\hat \rho)\right)$
ACG iterations, where $T_{ACG}(c,\hat \rho)$ is as in \eqref{TotalACG_perouter}.
Second, under the reasonable assumption that right-hand-side of the second bound in \eqref{def:cbar-eta-rho} is $\Omega(1)$, it is easy to show that 
\begin{align}
\label{eq:c_alpha_bd1}
& \hat{c}_{\alpha}(\hat{\rho},\hat{\eta}) =  
\Theta\left(\frac{\sqrt{m_f}}{\|A\|} \max \left\{\frac{\kappa_2^2 \sqrt{m_f}}{\hat{\eta}\|A\|}, 
\frac{\kappa_1}{\hat\rho}\right\}\right) 
\quad 
\text{if } \alpha= \Theta(1), \\
\label{eq:c_alpha_bd2}
& \hat{c}_{\alpha}(\hat{\rho},\hat{\eta}) =
\Theta\left(\frac{m_f}{\|A\|^2}\max\left\{\frac{\kappa_2^2}{\hat{\eta}},
\frac{\kappa_1^2}{\hat\rho^2}\right\}\right) 
\quad 
\text{if } \alpha=\Theta(1/c).
\end{align}
This remark together with the fact that $T_0(\hat \rho)=\mathcal{O}(\hat\rho^{-2})$ then imply that the ACG iteration complexity of S-IAIPAL, up to a logarithmic term, is
${\cal O}(\hat\rho^{-5/2}+\hat\rho^{-2}\hat\eta^{-1/2})$ when
$\alpha = \Theta(1)$,
and  ${\cal O}(\hat\rho^{-3}+\hat\rho^{-2}\hat\eta^{-1/2})$
when $\alpha={\Theta}(1/c)$. 
Third, the number of iterations where S-IAIPAL performs a full multiplier update (i.e., $p_k = q_k$) is ${\cal O}(\hat\rho^{-2}[\alpha c]^{-1})$. In particular, if $\alpha = \Theta(1)$ and $\hat\rho = \hat\eta$, then the number of full multiplier updates is ${\cal O}(\hat\rho^{-1})$ when $c= \Theta(\hat{c}_\alpha)$ and is ${\cal O}(\hat\rho^{-2})$ when $c= \Theta(1)$.
Fourth, 
since the threshold $\hat{c}_\alpha$ in \eqref{def:cbar} is not computable in practice, it is not clear how one can
choose a penalty parameter $c$ such that $\hat{c}_\alpha \le c= {\mathcal O}(\hat{c}_\alpha)$. 


The next subsection presents IAIPAL which repeatedly invokes S-IAIPAL with increasing penalty parameter values until a $(\hat \rho,\hat\eta)$--approximate solution of \eqref{optl0} is obtained. 
Moreover, it is shown that, up to a logarithmic term, the overall number of ACG iterations performed by this scheme is
the same as the one of S-IAIPAL under the condition $\hat{c}_\alpha \le c={\mathcal O}(\hat{c}_\alpha)$.

	\subsection{The IAIPAL method}\label{Sec:dynamic-IPAAL}
This subsection describes the IAIPAL method and  establishes its ACG iteration-complexity.

The statement of IAIPAL below
makes use of S-IAIPAL presented in
Subsection~\ref{Sec:static-IPAAL}. More specifically, it consists of  
repeatedly invoking S-IAIPAL
with $c=c_\ell:=c_1 \tau^{\ell-1}$
where $c_1$ is an initial choice
for the penalty parameter, $\tau>1$, and
$\ell$ is the
S-IAIPAL call count.

	\vgap
	\hrule
	\vspace*{0.5em}
	\noindent\quad{\bf IAIPAL}
	\vspace*{0.5em}
	\hrule
	\vspace*{0.5em}
	\begin{itemize}
		\item[(0)] Let  a quadruple of scalars  $(\nu,\sigma,\tau)\in \Re_{++}\times(0,1/\sqrt{2}]\times(1,+\infty)$ and  a pair of tolerances $(\hat \rho, \hat \eta)\in \Re_{++}\times\Re_{++}$ be given; choose 
		$\pen_1 > 0$ and set $\ell \gets 1$;
		\item[(1)] choose an initial point $z_0^{(\ell)} \in {\mathcal H}$ and some $\alpha_\ell \in\Re_{++}$; call 
		S-IAIPAL  with inputs $z_0 =z_0^{(\ell)}$, 
		$\nu$, $\sigma$, $\hat \rho$, $\hat \eta$, $c = c_\ell$, and 
		$\alpha = \alpha_\ell$;
 		\item[(2)] if 	S-IAIPAL successfully stops with a triple $(\hat z, \hat p,\hat w)$, then output this triple and stop; otherwise,  set $c_{\ell+1}\leftarrow \tau c_\ell$, set $\ell \gets \ell + 1$, and return to step~1.  
	\end{itemize}
	\vspace*{0.5em}
	\hrule
	\vgap
	
	We now make some remarks about IAIPAL.
	First, the initial point $z_0^{(\ell)}$ chosen in step 1 can either be the same point (cold start) across all S-IAIPAL calls or
	a varying point. 
	In the latter case, 
	a simple approach (warm start)
	is to choose $z_0^{(\ell)}$ as  the last iterate computed in the most recent call to S-IAIPAL. 
	Second, every outer iteration within the $\ell$-th S-IAIPAL call uses the penalty parameter $c_\ell = c_1 \tau^{\ell-1}$.
	Third, if $\ell$-th 
	S-IAIPAL call does not
	successfully stop in step 2 or, equivalently, declares $c_\ell$ small in step~3 of S-IAIPAL, then the next penalty parameter $c_{\ell+1}$ is increased by a multiplicative factor $\tau>1$.
	Finally, $\alpha_\ell$ can be chosen as a constant in every execution of step 1, or it can change. 
	For example, choosing $\alpha_\ell = 1/c_\ell$ guarantees that a Lagrange multiplier update is performed at every outer iteration of an S-IAIPAL call.

	
	The following result establishes the overall ACG
	iteration-complexity for
	 IAIPAL to obtain
	a  $(\hat\rho,\hat\eta)$-approximate stationary  solution of \eqref{optl0}.

	\begin{theorem}\label{mainTheo2} Assume that conditions {\bf (B1)}--{\bf(B5)}  of
		Subsection \ref{subsec:assump-Motivation} hold and define the scalar
		\begin{equation}
		\label{eq:c_hat_def}
		\hat c(\hat{\rho},\hat{\eta}) := \sup_{\ell \geq 1} \hat{c}_{\alpha_\ell}(\hat{\rho}, \hat{\eta}),
		\end{equation}
		where $\hat{c}_{\alpha_\ell}(\cdot, \cdot)$ is as in \eqref{def:cbar}.
Then, IAIPAL  obtains a $(\hat\rho,\hat\eta)$-approximate stationary  solution $(\hat z,\hat \lm,\hat w)$ of problem~\eqref{optl0} in
	\begin{equation}
		\label{mainbound}
		\mathcal{O}\left(T_{ACG}(\hat{c}(\hat{\rho},\hat{\eta}) + c_1,\hat \rho) \cdot \log_1^{+}\left\{\frac{\hat{c}(\hat{\rho},\hat{\eta})}{c_{1}}\right\} \right)
	\end{equation}
ACG iterations, where $c_1$ is the initial penalty parameter in IAIPAL, $\kappa_1$ and $\kappa_2$ are as in \eqref{def:kappa1-kappa2}, and $T_{ACG}(\cdot,\cdot)$ is as in \eqref{TotalACG_perouter}. 
\end{theorem}
	
	\begin{proof}
First note that the  $\ell$-th loop of IAIPAL invokes  S-IAIPAL with  penalty parameter $c_\ell=\tau^{\ell-1} c_1$, for every $\ell \ge 1$. 
It is easy to see that if IAIPAL  stops in its first call to S-IAIPAL,  then the statement of the theorem follows trivially in view of the stopping criterion in step~2 of IAIPAL and  Theorem~\ref{theor:StaticIPAAL}(b). Suppose then that IAIPAL calls S-IAIPAL more than once and let $\hat{c} = \hat{c}(\hat{\rho},\hat{\eta})$. Defining the integer
\begin{equation}\label{ineq:cl}	
\bar{\ell}:=\min\left\{ \ell:  c_\ell \geq \hat{c} \right\},
\end{equation}
it follows from Theorem~\ref{theor:StaticIPAAL}(d) that a $(\hat \rho,\hat\eta)$-approximate solution of \eqref{optl0} is obtained in at most $\bar{\ell}\geq 2$ calls to S-IAIPAL. In view of the minimality in \eqref{ineq:cl} and the penalty update rule in step~2 of IAIPAL, we have $c_{\bar \ell} \leq \tau \hat c$ and, hence,
		\begin{equation}\label{ineq:aux-clbar}
		\bar{\ell}=\log_{\tau}(\tau^{\bar{\ell}})=\log_{\tau}\frac{\tau c_{\bar{\ell}}}{c_{1}}\leq\log_{\tau}\frac{\tau^{2}\hat{c}}{c_{1}}=2+\log_{\tau}\frac{\hat{c}}{c_{1}}.
		\end{equation} 	
Combining \eqref{ineq:aux-clbar}, the fact that $T_{ACG}(\hat c, \hat\rho) \geq T_{ACG}(\hat c_{\alpha_\ell}, \hat\rho)$ for $\ell \geq 1$, and Theorem~\ref{theor:StaticIPAAL}(b), we conclude that the number of ACG iterations of IAIPAL is on the same order of magnitude as in \eqref{mainbound}.
\end{proof}

\vgap	
	
	
We now make some remarks about Theorem~\ref{mainTheo2}.
First, it is easy to see that for fixed $(\hat{\rho},\hat{\eta})$, it holds that $\sup_{\alpha > 0} \hat{c}_\alpha(\hat{\rho}, \hat{\eta})$ is finite and, hence, $\hat c$ in \eqref{eq:c_hat_def} is also finite.
Second, its iteration-complexity does not depend on
how $z_0$ is selected in step 0. As a consequence,
it applies to both the
cold start and
the warm start approaches mentioned above.
Third, it follows from Theorem \ref{mainTheo2} that the total number of ACG iterations  of  IAIPAL is, up to a logarithmic term, the same as that of S-IAIPAL  with  penalty parameter $c$ such that $\hat{c}(\hat{\rho},\hat{\eta}) \leq c={\mathcal O}(\hat{c}(\hat{\rho},\hat{\eta}))$.  

The next result describes \eqref{mainbound} only in terms of $(\hat{\rho}, \hat{\eta})$ for two choices of $\alpha_\ell$.
	
\begin{corollary}
Assume that conditions {\bf (B1)}--{\bf(B5)}  of
Subsection \ref{subsec:assump-Motivation} hold and that $\max\{c_1, c_1^{-1}\} = {\cal O}(\hat{c}(\hat{\rho},\hat{\eta}))$, where $\hat{c}(\cdot,\cdot)$ is as in \eqref{eq:c_hat_def}. Then, IAIPAL obtains a $(\hat\rho,\hat\eta)$-approximate stationary solution of problem~\eqref{optl0} in a number of ACG iterations/resolvent evaluations bounded, up to a logarithmic term, by
\begin{align}
& {\cal O}(\hat\rho^{-5/2}+\hat\rho^{-2}\hat\eta^{-1/2}) \quad \text{if } \alpha_\ell=\Theta(1), \label{eq:alpha_choice1} \\
& {\cal O}(\hat\rho^{-3}+\hat\rho^{-2}\hat\eta^{-1/2}) \quad \text{if } \alpha_\ell=\Theta(1/c_\ell). \label{eq:alpha_choice2}
\end{align}
Consequently, if IAIPAL performs a multiplier update at every outer iteration of S-IAIPAL, i.e., choose $\alpha_\ell=1/c_\ell$ in step~1 of IAIPAL, then its ACG iteration complexity is as in \eqref{eq:alpha_choice2}.
\end{corollary}

\begin{proof}
This follows immediately from Theorem~\ref{mainTheo2},  the definition of $T_{ACG}$ in \eqref{TotalACG_perouter}, and \eqref{eq:c_alpha_bd1}--\eqref{eq:c_alpha_bd2}.
\end{proof}

\section{Proof of Theorem~\ref{theor:StaticIPAAL}}\label{sec:technical}

The goal of this section is to provide the proof of Theorem~\ref{theor:StaticIPAAL}
which describes the main properties of S-IAIPAL.


We start by motivating the results developed in this section.
	A major part
	of our effort lies in showing that
	 the residual  and the
    feasibility gap sequences $\{\|\hat w_k\|\}$
    and $\{\|A \hat z_k-b\|\}$
    generated
    by S-IAIPAL with penalty parameter $c$ satisfy
    
    \begin{equation}\label{ineq:auxresidualbound}
    \min_{i \le k} \|\hat w_i\|={\cal O} \left(\frac{1}{\sqrt{k}} + \frac{1}{\sqrt{\alpha} c}
\right),\quad  \|A \hat z_k-b\|={\cal O}\left(\frac{1}{c}\right),\quad \forall k\geq 2.
    \end{equation}
    Observe that \eqref{ineq:auxresidualbound} implies that there exists a
    range of sufficiently large values of $c$ satisfying $c^{-1} = {\cal O}(\min\{\hat\rho \sqrt{\alpha},\hat \eta\})$ and such that
    S-IAIPAL finds a $(\hat\rho,\hat\eta)$-approximate stationary  solution of \eqref{optl0} in ${\cal O}(\hat \rho^{-2})$ S-IAIPAL iterations.
    Using this observation together with Proposition \ref{Prop:SIAIPAL-ACG},
    it is now easy to see that there exists a significantly large 
    range of $c$'s for which
    the total number of ACG iterations performed by S-IAPIAL is  $\mathcal{O}(\hat\rho^{-5/2}+\hat\rho^{-2}\hat \eta^{-1/2})$, up to a multiplicative logarithmic term.
    Lemma~\ref{lem:wkhat-Deltak}(b) below
    establishes a key inequality
    towards proving the first relation  in \eqref{ineq:auxresidualbound},
    and the paragraph following this lemma outlines
    how
    this inequality is used to establish \eqref{ineq:auxresidualbound}.

The first technical result below describes some
important properties about the sequence $\{(\hat z_k,\hat p_k, \hat w_k)\}$ computed in step~2 of S-IAIPAL as well as other related sequences
which are also used
in the analysis of S-IAIPAL.

\begin{proposition}\label{prop:approxsol}
The following statements hold:
\begin{itemize}
    \item[a)]
the triple $(\hat z_k,\hat p_k,\hat w_k)$ 
generated in
step~2 of S-IAIPAL and  the residual $r_k$ defined in \eqref{def:rk} satisfy
		\begin{gather}\label{incl-ineq:wkhat}
\hat w_k \in \nabla f(\hat z_k) + \partial h (\hat z_k)+A^*\hat p_k,\\
 \label{ineq:zkhat-zk}
\lam \|\hat w_k\|\leq \left(1+ 2 \nu \right)\|r_k\|, \qquad  \|\hat z_k-z_k\| \leq   \frac{\nu}{2 (\lambda L_c+1)}\|r_k\|,
		\end{gather}
	where $\nu $ and  $L_c$ are  as in \eqref{definition of sigma};
	
		\item[b)]  the quadruple $(z_k,p_k,w_k, \varepsilon_k)$
		where
		$w_k$ is defined as
		\begin{equation}\label{def:wk-newdeltak}
 w_k:=\frac{1}\lam \left[(\lambda L_c+1)(z_k-\hat z_k) +r_k\right]
\end{equation}
and
		$(z_k,q_k,\varepsilon_k)$ is computed in step~1 of S-IAIPAL,	satisfies
		\begin{gather}
		\label{inclusion:wk}
		w_k\in  \nabla f(z_k) +  \partial_{(\lam^{-1}\varepsilon_k)} h (z_k)+A^*q_k, \\
		\label{ineq:wk-zk-zkhat}
		\lambda\|w_k\|\leq \left(1+\nu \right)\|r_k\|, \qquad \varepsilon_k\leq \frac{\sigma_c^2\|r_k\|^2}{2}
		\end{gather}
		where $ \sigma_c$ is  as in \eqref{definition of sigma}.
\end{itemize}
		\end{proposition}
		
\begin{proof} First note that the last inequality in \eqref{ineq:wk-zk-zkhat} follows immediately from  the inequality in \eqref{GIPPinc_ineq} with $\tilde \sigma=\sigma_c$ in view of step~1. Note also that the quantities  $(\tilde g,\tilde h)$,  $(z,\varepsilon)$, and $\tilde L$ defined as
\begin{align}\label{def:auxghtilde}
\tilde g:=\lam [\mathcal{L}_c(\cdot,p_{k-1})-h]&-\langle v_k, \cdot-z_k\rangle+\frac{1}{2}\|\cdot -z_{k-1}\|^2, \quad \tilde  h:=\lambda h,\\ (z,\varepsilon)&:=(z_k,\varepsilon_k), \quad \tilde L:=\lambda L_c+1
\end{align}
satisfy the assumptions of Lemma~\ref{lem:approxsolreps},
in view of {\bf (B2)}, {\bf (B5)}, \eqref{lagrangian2}, \eqref{def:epsSubdiff}, \eqref{eq:defell}, \eqref{gradLips}, and the inclusion in  \eqref{GIPPinc_ineq}.
Observe also that \eqref{def:rk}, \eqref{def:pk}, and \eqref{def:auxghtilde}  imply that
$\tilde z$ and $\widetilde w$ in \eqref{eq:def_zhat} are equal to
$\hat z_k$ and $(\lambda L_c+1)(z_k-\hat z_k)$,
respectively, and 
$\nabla \tilde g(z_k)= 
\lam[ \nabla f(z_k)+A^*q_k]-r_k$.
Hence, it follows from the conclusion of Lemma~\ref{lem:approxsolreps} that
\begin{align}
\label{eq:auxinclusions_prop}
(\lambda L_c+1)(z_k-\hat z_k )+r_k  &  \in \lambda [ \nabla f(z_k) + A^*q_k] + \partial (\lam h) (\hat z_k), \\
(\lambda L_c+1)(z_k-\hat z_k) + r_k &  \in \lam [ \nabla f(z_k) + A^*q_k] + \partial_{\varepsilon_k} (\lam h) (z_k) \label{eq:auxinclusion_prop2},\\
 (\lambda L_c+1)\|(z_k-\hat z_k)\|&\leq \sqrt{2(\lambda L_c+1)\varepsilon_k}\label{eq:auxineq_prop}.   
\end{align}
Hence,  inclusion \eqref{incl-ineq:wkhat} follows from  
\eqref{def:wkhat},  \eqref{Relation:pkhat-pk}, \eqref{eq:auxinclusions_prop}, and a well-known property of the
$\varepsilon$-subdifferential of a function which
follows directly from its definition \eqref{def:epsSubdiff}. Moreover, inclusion  \eqref{inclusion:wk} follows immediately from  
\eqref{def:wk-newdeltak} and \eqref{eq:auxinclusion_prop2}.
The first inequality in \eqref{ineq:wk-zk-zkhat} follows from \eqref{def:wk-newdeltak}, the Cauchy-Schwarz inequality, \eqref{eq:auxineq_prop}, the last inequality in \eqref{ineq:wk-zk-zkhat}, and the definition of $\sigma_c$ in \eqref{definition of sigma}.
Now, \eqref{gradLips},  \eqref{def:wkhat},  \eqref{def:wk-newdeltak},  \eqref{eq:auxineq_prop}, the definition of $L_c$ in \eqref{definition of sigma}, and the Cauchy-Schwarz inequality, imply that 
 \[
\lambda \|\hat w_k\|\leq \|\lambda w_k\| + \lambda(L_f + c \|A\|^2)\|\hat z_k-z_k\| \le  \|\lambda w_k\|+\lambda \sqrt{2(\lambda L_c+1)\varepsilon_k}.
 \]
 The  first inequality in \eqref{ineq:zkhat-zk} then follows from the above inequalities together with \eqref{ineq:wk-zk-zkhat} and the definition of $\sigma_c$ in \eqref{definition of sigma}.
 Finally, the second inequality in \eqref{ineq:zkhat-zk} follows
 immediately from \eqref{eq:auxineq_prop}, the last inequality in \eqref{ineq:wk-zk-zkhat}, and the definition of $\sigma_c$ in \eqref{definition of sigma}. 
\end{proof}

	We now make two remarks about
	Proposition~\ref{prop:approxsol}.
	First,
	the residual $w_k$ in \eqref{def:wk-newdeltak} 
	does not appear in the description
	of S-IAIPAL (and hence IAIPAL), 
	but it plays an important
	role in its analysis.
	More specifically, the residual pair
	$(w_k,\varepsilon_k)$, and the corresponding bounds
	developed for it in \eqref{ineq:wk-zk-zkhat},
	play a crucial role in proving that the
	sequence $\{p_k\}$ of Lagrange multipliers is bounded.
Second, the right hand sides of the inequalities in
\eqref{ineq:zkhat-zk} and \eqref{ineq:wk-zk-zkhat} are all expressed in terms of $\|r_k\|$ since
a substantial part of our analysis will
concentrate on deriving suitable bounds for it,
and hence for the quantities which are
bounded in \eqref{ineq:zkhat-zk} and \eqref{ineq:wk-zk-zkhat}.

The following technical result  derives an estimate on  $\{\|r_k\|\}$ in terms of the variation of
the augmented Lagrangian function
 along the sequence $\{(z_k,p_k)\}$ and  the  variation of the sequence of Lagrangian multipliers $\{p_k\}$.

	\begin{lemma}\label{aux:lemma-auglagb22}
		Let  $\{(z_k,\lm_k,v_k,\varepsilon_k)\}$ be generated by S-IAIPAL, let   $\{r_k\}$ be as in \eqref{def:rk}, and define $\{\Delta \lm_k\}$ as 
		\begin{equation}\label{def:Delta}
		\Delta \lm_k:=\lm_k-\lm_{k-1}, \qquad \forall k\geq 1.
		\end{equation}
		Then, the following inequality  holds for every $k\geq 1$:
		\begin{equation}\label{auglagDecreasing2} 
		\|r_k\|^2\leq \frac{2\lambda}{1-\sigma_c^2}\left({\cal L}_\pen(z_{k-1},\lm_{k-1})-{\cal L}_\pen(z_k,\lm_{k})+\frac{1}{\pen} \|\Delta\lm_k\|^2\right).
		\end{equation}
	\end{lemma}	
	\begin{proof} 
	In view of the update rule for $p_k$ given in step~1 of S-IAIPAL and the definitions of  ${\cal L}_c$ and $\Delta p_k$ given in \eqref{lagrangian2} and \eqref{def:Delta}, respectively, we have
		\begin{equation}
		\label{dreasingproperty00}
		{\cal L}_\pen(z_k,\lm_k)-{\cal L}_\pen(z_k,\lm_{k-1})
		=\left\langle\Delta\lm_k,Az_k-b\right\rangle
		=\frac{1}{\pen} \|\Delta\lm_k\|^2,
		\end{equation}
		where the last identity follows from the fact that $\Delta p_k=0$ when $k \not\equiv 1 \bmod \lceil \alpha c \rceil$ and $A{z_k}-b= c^{-1} \Delta p_k$ when $k \equiv 1 \bmod \lceil \alpha c \rceil$.
		Now,  it follows from \eqref{def:epsSubdiff}, \eqref{GIPPinc_ineq}, and \eqref{def:rk} that
		\begin{align*}
		& \lambda{\cal L}_\pen(z_k,\lm_{k-1}) - \lambda{\cal L}_\pen(z_{k-1},\lm_{k-1})
		\leq -\frac{1}{2} \|z_{k}-z_{k-1}\|^2 +
		\inner{v_{k}}{z_k-z_{k-1}} +\varepsilon_k \\
		&= -\frac{1}{2} \|v_k+z_{k-1} - z_{k} \|^2 + \frac{\|v_k\|^2}2 + \varepsilon_k  \leq -\frac{1-\sigma_c^2}{2} \|r_k\|^2,
		\end{align*}
		which implies that

$$
\frac{1-\sigma_c^2}{2\lambda} \|r_k\|^2\leq		{\cal L}_\pen(z_{k-1},\lm_{k-1})-{\cal L}_\pen(z_k,\lm_{k-1}).
$$
The inequality in \eqref{auglagDecreasing2} then follows by combining  the  latter inequality with \eqref{dreasingproperty00}. 
	\end{proof}
	
	Recall that Proposition~\ref{prop:approxsol}(a) implies that the triple
	$(\hat z,\hat p, \hat w)=(\hat z_k,\hat p_k,\hat w_k)$
	satisfies the inclusion in \eqref{eq:approx_stationary1'}.
	The following technical result gives a preliminary
	bound on $\|\hat w_k\|$ and establishes the key
	inequality mentioned in the second paragraph
	of this section.


	\begin{lemma}\label{lem:wkhat-Deltak} Consider the sequences  $\{(z_k,\lm_k,v_k,\varepsilon_k)\}$ and $\{(\hat z_k,\hat p_k,\hat w_k)\}$ generated by  S-IAIPAL  and let  $C_1$, $\Delta_k$, and $\Delta p_k$ be as in  \eqref{definition of sigma}, \eqref{definition:Deltak}, and \eqref{def:Delta}, respectively.
Then,   the following statements hold: 
	\begin{itemize}
	\item[a)] for every $k\geq 1$, we have
		\begin{equation}\label{ineq:chi-DeltaL&DeltaPk}
		\|\hat w_k\|^2
		\leq \frac{C_1}{\lam}\left[{\cal L}_\pen(z_{k-1},\lm_{k-1})-{\cal L}_\pen(z_k,\lm_k) +\frac{1}{\pen} \|\Delta\lm_k\|^2\right];
		\end{equation}
	\item[b)]	for every $k\geq 2$, we have
		\begin{equation}
		\label{ineq:min-chik}
		\min_{2\leq i \le k} \|\hat w_i\|^2 \le  \frac{C_1}\lam \left[ \Delta_k + \frac{1}{k-1}\sum_{i=2}^k \frac{\|\Delta p_i\|^2}{c}  \right].
		\end{equation}
	\end{itemize}	
	\end{lemma}

\begin{proof}
a) 
	It follows from Proposition~\ref{prop:approxsol}(a) that
	the triple $(\hat z_k,\hat p_k,\hat w_k)$ computed in step 2 of S-IAIPAL
	satisfies, in particular, the first inequality in \eqref{ineq:zkhat-zk}.
	This conclusion together with
	inequality \eqref{auglagDecreasing2} then imply that 
		\begin{equation*}
		\|\hat w_k\|^2\le  \frac{\left(1+ 2\nu\right)^2\|r_k\|^2}{\lam^2}\leq \frac{2\left(1+ 2\nu\right)^2}{\lam (1-\sigma_c^2)}\left[{\cal L}_\pen(z_{k-1},\lm_{k-1})-{\cal L}_\pen(z_k,\lm_k) +\frac{1}{\pen} \|\Delta\lm_k\|^2\right],
		\end{equation*}
		and hence that \eqref{ineq:chi-DeltaL&DeltaPk} holds,
		in view of the definition of $C_1$   and the fact  $\sigma_c\leq \sigma$, see \eqref{definition of sigma}.

b) Summing inequality \eqref{ineq:chi-DeltaL&DeltaPk} from $k=2$ to $k=k$, and  using the definition of $\Delta_k$ given in \eqref{definition:Deltak}, we obtain
 		\begin{gather*}
		(k-1)\min_{2\leq i\leq k}\|\hat w_i\|^2 \leq \sum_{i=2}^k \|\hat w_i\|^2 
		\leq \frac{C_1}{\lambda}\left[(k-1) \Delta_k + \sum_{i=2}^k \frac{\|\Delta p_i\|^2}{c}\right].
		\end{gather*}
\end{proof}



We now outline how \eqref{ineq:min-chik}, together with some technical results below, can be used to establish the first bound in  \eqref{ineq:auxresidualbound}.
Bound \eqref{ineq:min-chik} on  $\min_{i\leq k}\|\hat w_i\|^2$ 
 is the sum of several
terms, one of which depends on $\Delta_k$.
Now, Lemmas~\ref{lem:constante0} and  \ref{lem:boundingLkbelow} show that
$\Delta_k$ is ${\cal O}([1+\|p_k\|^2] k^{-1})$.
Moreover, with the help of Lemmas~\ref{lem:aux-rk-deltak}--\ref{lem:auxtoboundpk},
Proposition~\ref{prop:mainprop-boundpk} establishes
that $\|p_k\|$ is  bounded by
a constant independent of $c$.
Note that \eqref{ineq:min-chik}, the above two observations, and the update rule \eqref{def:pk0} then imply that  $\min_{i\leq k}\|\hat w_i\|^2$
behaves as ${\cal O}(k^{-1}+\alpha^{-1}c^{-2})$ and, hence, that
the first relation in \eqref{ineq:auxresidualbound} holds.  

The next result, whose proof can be found in \cite[Lemma A.3]{RJWIPAAL2020},  will be used in the proof of Lemma~\ref{lem:constante0}.
\begin{lemma} \label{lem:auxNewNest2}
		Let proper function $\tilde \phi: \Re^n \to (-\infty,\infty]$,
		scalar $\tilde \sigma \in (0,1)$  and
		$(z_0,z_1) \in \Re^n \times     \dom \tilde \phi$ be given, and assume that there exists
		$(v_1,\varepsilon_1)$ such that
		\begin{gather}
		v_1 \in \partial_{\varepsilon_1} \left(\tilde\phi+\frac{1}{2}\|\cdot-z_0\|^2\right) (z_1), 
		\quad  \|v_1\|^2 + 2 \varepsilon_1 \leq \tilde\sigma^2 \|v+z_{0}-z_1\|^{2}. \label{Auxeq:prox_incl}  
		\end{gather}
		Then, for every $z\in \Re^n$ and $s>0$, we have
		\[
		\tilde \phi(z_1)+\frac{1}{2} \left[ 1 - \tilde\sigma^2 ( 1 + s^{-1}) \right]\|v_1+z_0 - z_1\|^{2}\le \tilde \phi(z) +\frac{s+1}{2} \|z-z_0\|^2.
		\]
	\end{lemma}

	The following technical result shows that
	${\cal L}_\pen(z_1,\lm_1)$  can be
	majorized by a scalar which
	does not depend on $c$. This fact,
	which
	is not immediately apparent from
	the definition of
		${\cal L}_\pen(\cdot,\cdot)$,
		plays an important role
		in showing that S-IAIPAL or
		IAIPAL can start from an arbitrary (and hence infeasible) point in ${\mathcal H}$.

	\begin{lemma}\label{lem:constante0}  The first quadruple $(z_1,\lm_1,v_1,\varepsilon_1)$  generated by  S-IAIPAL satisfies
		\begin{equation} \label{aux:ineqTz1p0}
		{\cal L}_\pen(z_1,\lm_1)\leq 3\left(\Delta\phi^*+2m_fD^2\right)+\phi_*,
		\end{equation}
		where  $\phi_*$ and $\Delta\phi^*$ are  as in \eqref{def:zeta}.
	\end{lemma}
\begin{proof}
	The fact that
		$(z_1,v_1,\varepsilon_1)$ satisfies \eqref{GIPPinc_ineq} with  $k=1$ and $\tilde \sigma=\sigma_c$,  Lemma~\ref{lem:auxNewNest2} with
		$s=1$ and
		$\tilde{\phi} =\lambda {\cal L}_\pen(\cdot,\lm_0)$, and condition {\bf (B3)},
		imply that	for every $ z\in {\mathcal H}$,
		\begin{align*}
		\lambda{\cal L}_\pen(z_1,\lm_0) + \frac{1-2\sigma_c^2}{2}\|r_1\|^2 &\leq   \lambda{\cal L}_\pen(z,\lm_0) +\|z-z_0\|^2 
		\leq   \lambda{\cal L}_\pen(z,\lm_0) + D^2,
		\end{align*} 
		where  $r_1$ is  as in \eqref{def:rk}.
		Using the definitions of $\phi^*$ and $\lambda$ given in
	\eqref{optl0} and \eqref{definition of sigma}, respectively, the fact that
	$1-2\sigma^2_c \ge 1-2\sigma^2 \ge 0$
	due to the definitions of $\sigma$ and $\sigma_c$ in step~0 of S-IAIPAL,
	and the fact that
	the definition of ${\cal L}_\pen$ in \eqref{lagrangian2}  implies that  $ {\cal L}_\pen(z,\lm_0)=(f+h)(z)$ for every  $z\in { \mathcal F}:=\{z\in {\mathcal H}:A z=b\}$, we then conclude from
	the above inequality, as $z$ varies in
	${\cal F}$, that
		\begin{equation*}
{\cal L}_\pen(z_1,\lm_0) \leq \phi^*+2m_fD^2.
		\end{equation*}
		The above inequality together with
		 the fact that  $p_0=0$, \eqref{def:pk} with $k=1$, and the definitions of ${\mathcal L_c}$ and $\phi_*$ given in \eqref{lagrangian2} and \eqref{def:zeta}, respectively, then imply that
		\begin{align*}
	    {\cal L}_\pen(z_1,\lm_1) &=  {\cal L}_\pen(z_1,\lm_0)+ \pen\|Az_1-b\|^2 =
	    3 {\cal L}_\pen(z_1,\lm_0)
	    - 2 (f+h)(z_1) \\
	    & \leq 3(\phi^*+2m_fD^2) -2 \phi_*,
	    \end{align*}
	  which  proves \eqref{aux:ineqTz1p0} in view of the definition of $\Delta \phi^*$.
	    \end{proof}

	The following technical result shows that 
	$\Delta_k = {\cal O}([1+c^{-1}\|p_k\|^2]k^{-1})$.
	
	\begin{lemma}\label{lem:boundingLkbelow} Let  $\{(z_k,\lm_k)\}$ be  generated by  S-IAIPAL  and consider $\{\Delta_k\}$ as in \eqref{definition:Deltak}. Then, the following statements hold:
\begin{itemize}
    \item[a)] for every $k\geq 1$, we have
    \begin{equation}\label{ineq:p1po}
	{\cal L}_\pen(z_k,\lm_k)+\frac{\|\lm_k\|^2}{2 c }\geq \phi_*,
    \end{equation}
where   $\phi_*$ is as in \eqref{def:zeta};
    \item[b)] for every $k\geq 2$, we have 
    \begin{equation}
        \label{ineq:bound-Deltak}
		\Delta_k\leq \frac{1}{k-1}\left[3\left(\Delta\phi^*+2m_fD^2\right) + \frac{\|p_k\|^2}{2c}\right],
		\end{equation}
where $\Delta\phi^*$ is as in  \eqref{def:zeta}.
	\end{itemize}		
	\end{lemma}
	\begin{proof} (a) Using the definitions of ${\cal L}_\pen$ and $\phi_*$ given in \eqref{lagrangian2} and \eqref{def:zeta}, respectively, we have
		\begin{align*}
		{\cal L}_\pen(z_k,\lm_k) &=  (f+h)(z_k)  +  \Inner{\lm_k}{Az_k-b}+\frac{c}2 \|Az_k-b\|^2\\
		&\geq \phi_*+\frac{1}{2}\left\|\frac{\lm_k}{\sqrt{c}} + \sqrt{c}(Az_k-b)\right\|^2- \frac{1}{2c} \|\lm_k\|^2,
		\end{align*}
		and hence that \eqref{ineq:p1po} holds.

(b) This statement follows  from
\eqref{aux:ineqTz1p0}, \eqref{ineq:p1po}, and the definition of $\Delta_k$ in \eqref{definition:Deltak}.
	\end{proof}

	The next technical results
	(i.e., Lemmas \ref{lem:aux-rk-deltak}--\ref{lem:auxtoboundpk})
	develop the necessary tools for showing
	in Proposition~\ref{prop:mainprop-boundpk}
	that the sequence $\{p_k\}$ is bounded.
	The first one gives some straightforward bounds among the
	different quantities involved in the analysis of S-IAIPAL.

	\begin{lemma}\label{lem:aux-rk-deltak} Let $\{(z_k,\lm_k,v_k,\varepsilon_k)\}$  be generated by S-IAIPAL and let $\{r_k\}$ be as in \eqref{def:rk}.
		Then, the following inequalities hold for every $k \ge 1$,
		\begin{equation}\label{ineq:vk-epsk}
		\quad \|r_k\|\leq \frac{D}{1-\sigma},\quad  \|v_k\|^2\leq C_2 D^2, \quad \varepsilon_k\leq \frac{C_2D^2}{2},
		\end{equation}
		where $D$ is as in {\bf (B3)} and $\sigma$ is as in step~0 of S-IAIPAL, and $C_2$ is as in \eqref{definition:C3}.
	\end{lemma}
	\begin{proof}
First note that, in view of step~1 of S-IAIPAL, the tuples  $(\lam,z_{k-1},p_{k-1})$ and $(z_k,v_k,\varepsilon_k)$
	satisfy \eqref{GIPPinc_ineq}. Hence, using the inequality in \eqref{GIPPinc_ineq}, the definition of $r_k$ given in \eqref{def:rk},  the triangle inequality, the first condition in {\bf (B3)}, and the fact that $\sigma_c\leq \sigma$, we have
		\begin{equation}\label{aux:ineq1234}
	\|r_k\|-D\leq 	\|r_k\| -\|z_k-z_{k-1}\| \leq \|v_k \| \le \sigma \|r_k\|, \qquad \varepsilon_k\leq \frac{\sigma^2\|r_k\|^2}{2}. 
		    \end{equation}
		The first inequality in \eqref{ineq:vk-epsk} immediately follows from the first setting of inequalities  in \eqref{aux:ineq1234}. The last two inequalities in \eqref{ineq:vk-epsk} follow from the first inequality in \eqref{ineq:vk-epsk},   the  last two inequalities in \eqref{aux:ineq1234} and the definition of $C_2$ in \eqref{definition:C3}.
	\end{proof}
	
	The  following basic result is used in Lemma~\ref{prop:bounding-sumpk}. Its proof can be found, for instance, in \cite[Lemma~A.4]{MaxJeffRen-admm}. Recall that $\sigma^+_A$ denotes  the smallest positive singular value of a nonzero linear operator $A$ . 

		\begin{lemma}\label{lem:linalg} 
Let $A:\Re^n \to \Re^l$ be a  nonzero linear operator. Then,
$\sigma^+_A\|u\|\leq \|A^*u\|$,   for every $u \in A(\Re^n)$.
\end{lemma}

The next result defines a slack $\xi_k \in \partial_{(\lambda^{-1}\varepsilon_k)} h(z_k)$ which realizes
	the inclusion in \eqref{inclusion:wk} and gives a preliminary
	bound on $\|p_k\|$ in terms of $\|\xi_k\|$.

	\begin{lemma}\label{prop:bounding-sumpk} 
	Consider the sequence $\{(z_k, q_k, v_k,\varepsilon_k)\}$ generated by S-IAIPAL and the sequence $\{w_k\}$
	as in \eqref{def:wk-newdeltak}, and define 
	\begin{equation}
	\label{def:xitildek}
	\xi_k:= w_k - \nabla f( z_k) - A^*q_k
\end{equation}
for every $k \ge 1$.
		Then, the following statements hold:
		\begin{itemize}
		    \item[a)]
		    for every $k \ge 1$, we have
		    \begin{equation}\label{eq:inc-vtilde-boundBtilde}
		 \xi_k\in \partial_{(\lambda^{-1}\varepsilon_k)} h (z_k), \qquad  \|w_k\| \leq \frac{C_3D}{\lambda}
		\end{equation}
		where  $D$ is as in {\bf (B3)} and
		$C_3$ is as in \eqref{definition:C3}; 
		\item[b)] for every $k\geq 1$, we have 
		\begin{equation}
		\label{ineq:pk-xik}
		\sigma^+_A\|q_k\|\leq \|\xi_k\|+ \nabla_f +\frac{C_3D}{\lambda},
		\end{equation}
		where  $\nabla_f$ is
		as in {\bf (B3)}.
		\end{itemize}
		\end{lemma}
		\begin{proof}
	(a) 
	The inclusion in \eqref{eq:inc-vtilde-boundBtilde}
	follows from
	\eqref{inclusion:wk} and the definition of $\xi_k$ in \eqref{def:xitildek}.
	The inequality in \eqref{eq:inc-vtilde-boundBtilde}
	follows from  the first inequalities  in
	\eqref{ineq:wk-zk-zkhat} and  \eqref{ineq:vk-epsk}, and the definitions of $\sigma_c$ and $C_3$ in \eqref{definition of sigma} and \eqref{definition:C3}, respectively.
		
	(b) Using {\bf (B4)}, the fact that $p_0=0$  together with the update formula for $q_k$ and $p_k$, it is easy to see that  $\{q_k\}\subset A(\Re^n)$.
	Using Lemma~\ref{lem:linalg},	relation \eqref{def:xitildek}, the triangle inequality, {\bf (B3)},   and the inequality in \eqref{eq:inc-vtilde-boundBtilde}, we conclude that
	    	\begin{align}
		\label{eq:qk_tech_bd}
		\sigma^+_A\|q_k\|  & \leq \|A^*q_k \| \leq \|\xi_k\|+\|\nabla f(z_k)\|+\|w_k\| \leq \|\xi_k\|+ \nabla_f +\frac{C_3D}{\lambda},
		\end{align}
			and, hence, that \eqref{ineq:pk-xik} holds.
\end{proof}

The next technical result essentially allows us to obtain a preliminary bound on
$\|\xi_k\|$ under assumption {\bf(B4)}. It is worth mentioning its proof
is based on a key inequality that appears in
the proof of Lemma~3 of 
\cite{PPmetNonconvex2019}.

	\begin{lemma}\label{lem:bound_xiN}
		Let $h$ be a function as in {\bf (B2)}.
		Then, for every $ z,z'\in {\mathcal H}$, $\varepsilon>0$,  and $\xi \in \partial_{\varepsilon} h(z)$, we have
		    \[
		\|\xi\|{\rm dist}_{\partial {\mathcal H}}(z') \le \left({\rm dist}_{\partial {\mathcal H}}(z')+\|z-z'\|\right)L_h + \inner{\xi}{z-z'}+\varepsilon,
		    \]
	where $\partial {\cal H}$ denotes the boundary of ${\cal H}$.
		  \end{lemma}
	\begin{proof}
 Let $\varepsilon>0$, $z,z' \in {\cal H}$ and $\xi \in \partial_{\varepsilon} h(z)$ be given. It follows from the Lipchitz continuity of $h$ in ${\bf (B2)}$ combined with  the equivalence between (a) and (d) of Lemma~\ref{Prop:eps-subd-bounded-cone}  that  
	 there exist $\xi_1\in \bar{B}(0,L_h)$ and $\xi_2\in  N^\varepsilon_{\mathcal H}(z)$ such that $\xi=\xi_1+\xi_2$.  Clearly,
	 it follows from the definitions of
	 $\bar{B}(0,L_h)$ and
	 $N_{\mathcal H}^{\varepsilon}(z)$
	 in Subsection \ref{sec:bas} that
		$$
\|\xi_1\| \le L_h, \quad {\mathcal H} \subset H_-:=\{u \in \Re^n 
:\inner{\xi_2}{u-z}-\varepsilon \leq 0\}.
		$$
		Using the last inclusion and
		the fact that $z'\in  {\mathcal H}$, 
		we easily see that
		$$
		{\rm dist}_{\partial {\mathcal H}}(z') \| \xi_2\| \leq {\rm dist}_{\partial H_-}(z')\|\xi_2\|=\inner{\xi_2}{z- z'}+\varepsilon.
		$$
The last inequality, the fact that $\xi=\xi_1+\xi_2$, the triangle inequality, and the Cauchy-Schwarz inequality, then imply that
		\begin{align*}
		{\rm dist}_{\partial {\mathcal H}}(z')\|\xi \|&\leq {\rm dist}_{\partial {\cal H}}(z')\|\xi_1\|+{\rm dist}_{\partial {\cal H}}(z')\|\xi_2\| \leq {\rm dist}_{\partial {\cal H}}( z')\|\xi_1\|+ \inner{\xi_2}{z- z'}+\varepsilon\\
		&= {\rm dist}_{\partial {\cal H}}(z')\|\xi_1\|-\inner{\xi_1}{z- z'}+ \inner{\xi}{z-z'}+\varepsilon\\
		&\leq  \left({\rm dist}_{\partial {\cal H}}(z')+\|z-z'\|\right)\|\xi_1\|+ \inner{\xi}{z-z'}+\varepsilon,
		\end{align*}
		which combined with the fact that $\|\xi_1\|\leq L_h$ shows that the conclusion of  the lemma holds. 
	\end{proof}

	The next lemma presents some important technical inequalities using the bounds in Lemma~\ref{lem:bound_xiN} and Lemma~\ref{prop:bounding-sumpk}(b).

\begin{lemma}\label{lem:auxtoboundpk}
 The iterates $\{(p_k, q_k, z_k)\}$ generated by S-IAIPAL satisfy:
  \begin{itemize}
     \item[a)] $\bar{d} \sigma_{A}^+ \|q_k\| \leq D\kappa_0 - \left\langle q_{k},Az_{k}-b\right\rangle$ for every $k\geq 1$;
     \item[b)] $c\|Az_k - b\| \leq \theta_D \kappa_0 (\sigma_A^+)^{-1} + \|p_{k-1}\|$ for every $k\geq 1$;
     \item[c)] $c^{-1}\|p_k\|^2+\bar{d}\sigma_A^+\|p_k\| \leq c^{-1}\inner{p_k}{p_{k-1}}+D\kappa_0$ for every $k \equiv 1 \bmod k_c$;
 \end{itemize}
		where $\sigma_A^+$ is defined in
		Subsection~\ref{sec:bas} and $\bar d$, $\theta_D$, and $\kappa_0$  are as in  \eqref{def:zeta}, \eqref{eq:condNumb}, and \eqref{def:kappa00}, respectively. 
		\end{lemma}
	\begin{proof}
	(a) Let $\{\xi_k\}$  be as in \eqref{def:xitildek}. 
	 Using \eqref{ineq:vk-epsk}, \eqref{eq:inc-vtilde-boundBtilde}, \textbf{(B3)}, the Cauchy-Schwarz and triangle inequalities, and the fact that $\lam=1/(2m_f)$ and $\|z_k - \bar{z}\|\leq D$, we first have that
\begin{align*}
\left\langle \xi_{k},z_{k}-\bar{z}\right\rangle +2m_{f}\varepsilon_{k} & \overset{\eqref{def:xitildek}}{=} \left\langle w_k -\nabla f(z_k) - A^*q_k,z_{k}-\bar{z}\right\rangle +2m_{f}\varepsilon_{k}\\ 
& \leq-\left\langle A^{*}q_{k},z_{k}-\bar{z}\right\rangle +\|z_{k}-\bar{z}\|\left(\|w_{k}\|+\|\nabla f(z_{k})\|\right)+2m_{f}\varepsilon_{k} \\
 & \leq-\left\langle q_{k},Az_{k}-b\right\rangle +D\left(\nabla_{f}+\left[2C_{3}+C_{2}\right]m_{f}D\right).
\end{align*}
Now, recall that $\bar d = {\rm dist}_{\partial \mathcal H}(\bar z)$  and note that $\xi_k \in \partial_{(\lambda^{-1}\varepsilon_k)} h(z_k)$ for every $k\geq 1$, in view of
\eqref{def:zeta} and Lemma~\ref{prop:bounding-sumpk}(a), respectively. 
Hence, using the above technical bound, Lemma~\ref{prop:bounding-sumpk}(b), Lemma~\ref{lem:bound_xiN} with $(\xi, z, z', \varepsilon) = (\xi_k, z_k, \bar z, \lambda^{-1}\varepsilon_k)$, 
the fact that $\lam=1/(2m_f)$, $\bar{d}\leq D$, and $\|z_k - \bar{z}\|\leq D$, and the definition of $\kappa_0$, we conclude that
\begin{align*}
\bar{d}\sigma_{A}^{+}\|q_{k}\| & \leq\bar{d}\left(\|\xi_{k}\|+\nabla_{f}+2m_{f}C_{3}D\right)\\
 & \leq(\bar{d}+\|z_{k}-\bar{z}\|)L_{h}+\left\langle \xi_{k},z_{k}-\bar{z}\right\rangle +2m_{f}\varepsilon_{k}+\bar{d}\left(\nabla_{f}+2m_{f}C_{3}D\right)\\
 & \leq D\left(2\left[L_{h}+\nabla_{f}\right]+\left[4C_{3}+C_{2}\right]m_{f}D\right)-\left\langle q_k,Az_{k}-b\right\rangle\\
 & =D\kappa_{0}-\left\langle q_{k},Az_{k}-b\right\rangle.
\end{align*}

b) Using part a), the definition of $q_k$, and the Cauchy-Schwarz and triangle inequalities, we have that
\begin{align*}
c\bar{d}\sigma_{A}^{+}\|Az_{k}-b\| & =\bar{d}\sigma_{A}^{+}\|q_{k}-p_{k-1}\|\leq\bar{d}\sigma_{A}^{+}\|q_{k}\|+\bar{d}\sigma_{A}^{+}\|p_{k-1}\|\\
 & \leq D\kappa_{0}-\left\langle q_k ,Az_{k}-b\right\rangle+\bar{d}\sigma_{A}^{+}\|p_{k-1}\| 
 \\ & \overset{\eqref{def:pk}}{=} D\kappa_{0}-\left\langle p_{k-1} ,Az_{k}-b\right\rangle -c\|Az_{k}-b\|^{2} +\bar{d}\sigma_{A}^{+}\|p_{k-1}\| \\
 & \leq D\kappa_{0}-c\|Az_{k}-b\|^{2}+\|p_{k-1}\|\left(\|Az_{k}-b\|+\bar{d}\sigma_{A}^{+}\right).
\end{align*}
Moving the $-c\|Az_{k}-b\|^{2}$ term to the left-hand-side, dividing the resulting inequality
by $\|Az_{k}-b\|+\bar{d}\sigma_{A}^{+}$, and using the definition of $\theta_D$ in \eqref{eq:condNumb} we conclude that
\[
c\|Az_{k}-b\|\leq\frac{D\kappa_{0}}{\|Az_{k}-b\|+\bar{d}\sigma_{A}^{+}}+\|p_{k-1}\|\leq\frac{\theta_{D}\kappa_{0}}{\sigma_{A}^{+}}+\|p_{k-1}\|.
\]

c) Let $k\equiv 1 \bmod k_c$. Using part a) and the fact that $q_k = p_{k} = p_{k-1} + c(Az_k - b)$, we have that
\begin{align*}
\bar{d}\sigma_{A}^{+}\|p_{k}\| & =\bar{d}\sigma_{A}^{+}\|q_{k}\|\leq D\kappa_{0}-\left\langle q_{k},Az_{k}-b\right\rangle =D\kappa_{0}+\frac{1}{c}\left\langle p_{k},p_{k-1}\right\rangle -\frac{1}{c}\|p_{k}\|^{2}
\end{align*}
which implies the desired bound.
\end{proof}

We observe that
	Lemma~\ref{lem:auxtoboundpk}(c)
	always holds under the weaker assumption
	that $\bar z \in {\cal H}$ and $A\bar z=b$
	but
	the scalar $\bar d$ which appears
	on it becomes zero
	when $\bar z \in \partial {\cal H}$.
	The following technical result establishes the boundedness of the sequence of Lagrange multipliers $\{p_k\}$ if instead {\bf (B4)}
	is assumed, and hence $\bar d>0$.

	\begin{proposition}\label{prop:mainprop-boundpk} The sequence $\{p_k\}$ generated by S-IAIPAL  satisfies
	\begin{equation} \label{ineq:pkbounded}
	\|p_k\|\leq \frac{\theta_D\kappa_0}{\sigma^+_A}
	 \end{equation}
	 for every $k\geq 0$, where $\kappa_0$ and $\theta_D$  are  as in \eqref{def:kappa00} and \eqref{eq:condNumb}, respectively.
	\end{proposition}
\begin{proof}
The proof is done  by induction on $k$.  Since $p_0=0$ and $\kappa_0\geq 0$, \eqref{ineq:pkbounded} trivially holds for $k=0$. Assume now that \eqref{ineq:pkbounded} holds with
$k=k-1$ for some $k\geq 1$. If $k\not\equiv 1 \bmod \lceil \alpha c\rceil$, then \eqref{def:pk} implies $p_k=p_{k-1}$, and \eqref{ineq:pkbounded} holds by our induction hypothesis. If $k\equiv 1 \bmod \lceil \alpha c\rceil$,
then the induction hypothesis together with Lemma~\ref{lem:auxtoboundpk}(c), the definition of $\theta_D$ in \eqref{eq:condNumb},  and the Cauchy-Schwarz inequality, imply that
		\begin{align*}
		\left(\frac{\|p_{k}\|}{c}+\sigma_A^+\bar{d}\right)\|p_{k}\| &\leq \frac{\|p_{k}\|\|p_{k-1}\|}{ c}+D\kappa_0\leq \frac{\|p_{k}\|D\kappa_0}{ c\sigma^+_A\bar d}+D\kappa_0 \\
		&=\left(\frac{\|p_{k}\|}{c}+\sigma^+_A\bar{d}\right)\frac{\theta_D\kappa_0}{\sigma^+_A}, 
		\end{align*}
		and hence that $\|p_{k}\|\leq \theta_D\kappa_0/\sigma^+_A$.
		We have thus proved that
	 \eqref{ineq:pkbounded} holds for all $k \ge 0$. 
	\end{proof}

The following result establishes that $\|\hat w_k\|=\mathcal{O}(\sqrt{\Delta_k} + \alpha^{-1/2} c^{-1})$ and $\|A\hat z_k-b\|=\mathcal{O}(c^{-1})$. Since $\Delta_k=\mathcal{O}(k^{-1})$ in view of \eqref{ineq:bound-Deltak} and \eqref{ineq:pkbounded}, it follows that $\|\hat{w}_k\|$ can be made arbitrarily small as the penalty parameter $c$ increases.   
	
\begin{lemma}
\label{lem:main_resid_bds}
The sequence 
		$\{(\hat z_k,\hat w_k, z_k)\}$ generated by  S-IAIPAL satisfies the following bounds:
	\begin{itemize}
	    \item[a)] $\|A\hat z_k-b \| \leq  {\kappa_2^2m_f}/{(c\|A\|^2)}$ for every $k\geq 1$;
	    \item[b)] $
	    \min_{2\leq i \le k} \|\hat w_i\|^2 \le 2 m_f C_1 \Delta_k+ 2{m_f\kappa_1^2}/{(c \lceil c \alpha \rceil \|A\|^2)}
	    $ for every $k\geq 1$.
	\end{itemize}
    where  $C_1$ is as in  \eqref{definition of sigma}, $\Delta_k$ is  as in \eqref{definition:Deltak}, $\kappa_0$ is as in \eqref{def:kappa00}, and  $\kappa_1$ and $\kappa_2$ are as in \eqref{def:kappa1-kappa2}.
		\end{lemma}
	\begin{proof}

a) It follows from Lemma~\ref{lem:auxtoboundpk}(b),  the second inequality in \eqref{ineq:zkhat-zk}, the  triangle inequality, and  the definitions  of $(\sigma_c,L_c)$ and $p_k$ given in  \eqref{definition of sigma} and \eqref{def:pk}, respectively, that
	\begin{align*}	
	\|A\hat z_k-b \|
	&  \leq  \|A z_k-b\| +  \|A\|\|\hat z_k-z_k\|
	\leq \frac{\theta_D\kappa_0}{\sigma^+_A c} + \frac{\|p_{k-1}\|}{c} + \frac{\sigma_c\|A\|\|r_k\|}{\sqrt{\lambda L_c+1}} \\
	&  \leq \frac{2\theta_D\kappa_0}{\sigma^+_A c} +\frac{\nu \|A\|\|r_k\|}{\lambda L_c+1} \le
	\frac{2\theta_D\kappa_0}{\sigma^+_A c} +\frac{\nu \|r_k\|}{\lambda c\|A\| },
	\end{align*}
where the last inequality is due to  $L_c\geq c\|A\|^2$. It follows from the above inequalities, \eqref{ineq:pkbounded}, and the first inequality in \eqref{ineq:vk-epsk}  that 
\begin{equation}\label{ineq:aux-control-infeasibility1}
	\frac{c\|A\|^2}{m_f} \|A\hat z_k-b \|\leq \frac{1}{m_f}\left(\frac{2\|A\|^2\theta_D\kappa_0}{\sigma^+_A}+\frac{\nu \|A\|D}{\lambda(1-\sigma)}\right)=2\theta_A\theta_D\left(\frac{\|A\|\kappa_0}{m_f}+ \frac{\nu \sigma_A^+\bar d}{1-\sigma}\right)
\end{equation}	
where the last relation is due to  the definitions of $\lambda$,  $\theta_A$,  and   $\theta_D$ given in \eqref{definition of sigma} and  \eqref{eq:condNumb}.
On the other hand, using the definitions of  $C_3$ and $\kappa_0$ given in  \eqref{definition:C3} and \eqref{def:kappa00}, respectively, and  the fact that   $\bar{d}\leq D$ and $\sigma_A^+\leq \|A\|$, we have  
\[
\frac{ \nu\sigma_A^+\bar d}{1-\sigma}\leq   C_3\|A\|D\leq \frac{\|A\| \kappa_0}{4m_f}.
 \]
Hence,  the desired bound  immediately follows from \eqref{ineq:aux-control-infeasibility1}, the latter inequalities and  the definition of $\kappa_2$  in \eqref{def:kappa1-kappa2}.

b) Define $I(k) := \{i : p_i \neq p_{i-1}, 2 \leq i \leq k\}$. In view of the multiplier update rule of S-IAIPAL, it is straightforward to show that $|I(k)| \leq \lfloor(k-1)/ \lceil \alpha c \rceil \rfloor \leq 2(k-1)/ \lceil \alpha c \rceil$.
Using   \eqref{eq:condNumb}, \eqref{ineq:min-chik},  \eqref{ineq:bound-Deltak}, \eqref{ineq:pkbounded}, the relation $(a+b)^2\leq 2a^2 + 2b^2$ for $a,b\in\r$, and the previous bound on $|I(k)|$, we have
\begin{align*}
& \frac{\lambda}{C_{1}}\min_{2\leq i\leq k}\|\hat{w}_{i}\|^{2} \leq 
\Delta_{k}+\sum_{i=2}^k\frac{\|\Delta p_{i}\|^{2}}{c(k-1)}
= \Delta_{k}+\sum_{i\in I(k)}\frac{\|\Delta p_{i}\|^{2}}{c(k-1)}\\
 & \leq\Delta_{k}+\frac{2\sum_{i\in I(k)}(\|p_{i}\|^{2}+\|p_{i-1}\|^{2})}{c(k-1)} \leq \Delta_{k}+ \left[\frac{\theta_{D}\kappa_{0}}{\sigma_{A}^{+}}\right]^{2}\left[\frac{4|I(k)|}{c(k-1)}\right] \\
 & \leq \Delta_{k}+ \frac{8}{c \lceil c \alpha \rceil} \left[\frac{\theta_{D}\kappa_{0}}{\sigma_{A}^{+}}\right]^{2}
 =\Delta_{k} + \frac{8\theta_{A}^2\theta_{D}^2\kappa_{0}^2}{c \lceil c \alpha \rceil \|A\|^2} = \Delta_{k} + \frac{\kappa_{1}^{2}}{c \lceil c \alpha \rceil \|A\|^{2}C_{1}},
\end{align*}
where the last relation is due to the definition of $\kappa_1$ given in  \eqref{def:kappa1-kappa2}. The desired bound then follows in view of the fact that $\lambda=1/(2m_f)$.
\end{proof}

Notice that, unless $c$ is sufficiently large, the bounds derived in the above lemma does not guarantee that either the feasibility residual $\|A\hat{z}_k-b\|$ or the stationarity residual $\|\hat{w}_k\|$ become sufficiently small, regardless of how large $k$ is. This is in contrast to all of the penalty/AL methods in \cite{WJRproxmet1, WJRComputQPAIPP, MinMax-RenWilliam, ImprovedShrinkingALM20, PPmetNonconvex2019, RJWIPAAL2020, inexactAugLag19} where the stationarity residual is always sufficiently small whenever the penalty parameter is updated.


	

	





	
We are now ready to present the proof of 	Theorem~\ref{theor:StaticIPAAL}.

\begin{proof}[Proof of Theorem~\ref{theor:StaticIPAAL}]
a) This statement follows immediately from \eqref{incl-ineq:wkhat}.

b) Let $T_0=T_0(m_f,\hat \rho)$ where $T_0(\cdot,\cdot)$ is as in \eqref{eq:outernumberiter} and assume that
S-IAIPAL has reached the $T_0$-th iteration and has not
stopped in its step 2.
Using \eqref{ineq:bound-Deltak} with $k=T_0$, 
and the definitions of $\lambda$, $\kappa_1$, and $T_0$ given in \eqref{definition of sigma},  \eqref{def:kappa1-kappa2},  and \eqref{eq:outernumberiter},  respectively, we then conclude that
		\begin{align*}
		\Delta_{T_0 } & \leq \frac{1}{{T_0}-1}\left[{3\left(\Delta\phi^*+2m_fD^2\right)} +  \frac{\|p_{T_0}\|^2}{2c}\right] \\
		& \leq \frac{1}{{T_0}-1} \left[{3\left(\Delta\phi^*+2m_fD^2\right)} +  \frac{(\theta_A\theta_D \kappa_0)^2}{2m_f} \right] \\
		& = \frac{1}{{T_0}-1} \left[{3\left(\Delta\phi^*+2m_fD^2\right)} +  \frac{\kappa_1^2}{16 C_1 m_f} \right] \leq \frac{\hat \rho^2}{4C_1m_f} =
\frac{\lambda\hat\rho^2}{2C_1}.
		\end{align*}
		
 Hence, S-IAIPAL must stop in step 3 of the $T_0$-th iteration.
	
	c) This statement follows immediately from Lemma~\ref{lem:main_resid_bds}(b) and  condition  \eqref{def:cbar-eta-rho} on the penalty parameter~$c$.
		
	d) First note that it follows from  part b) that S-IAIPAL stops either in step~2 or step~3 after a finite number of iteration. Now, in view of the stopping criterion in step~2 and part  c), it follows that
	S-IAIPAL stops with success at the $k$-th  iteration if and only if 
	$\hat w_k$ satisfies $\|\hat w_k\| \le \hat \rho$, in
	which case the triple
	$(\hat z_k,\hat p_k,\hat w_k)$ is a $(\hat\rho,\hat\eta)$-approximate stationary solution of \eqref{optl0} due to part a) and Definition~\ref{def:stationarypoint}. 
	Now, assume
	for contradiction that S-IAIPAL
	stops in step 3 (instead of
	step 2) at some
	iteration $k$, and hence that
	\[
	\min_{2\leq i \le  k} \|\hat w_i\| > \hat \rho, \qquad
	\Delta_k \le \frac{\lam \hat \rho^2}{2C_1} =
	\frac{ \hat \rho^2}{4m_fC_1}
	\]
	in view of the last observation above,
	\eqref{definition:Deltak}, and  the definition of $\lambda$ in \eqref{definition of sigma}.
These two inequalities together with
  \eqref{def:cbar-eta-rho}, and Lemma~\ref{lem:main_resid_bds}(c),
 yield the contradiction
\[
\hat\rho^2<\min_{2 \leq i \le k} \|\hat w_i\|^2 \le 2C_1m_f\Delta_k+ \frac{2m_f\kappa_1^2}{c \lceil c \alpha \rceil \|A\|^2}\leq \frac{\hat\rho^2}{2}+\frac{\hat\rho^2}{2}=\hat\rho^2.
\]
which must mean that S-IAIPAL stops with a $(\hat\rho,\hat\eta)$-approximate stationary solution in step~2.
\end{proof}

\section{Numerical experiments}

\label{sec:numerical}

This section presents experiments\footnote{See \href{https://github.com/wwkong/nc_opt/tree/master/tests/papers/IAIPAL}{https://github.com/wwkong/nc\_opt/tree/master/tests/papers/IAIPAL} for the full code.} which benchmark different variants of IAIPAL.
The first subsection benchmarks IAIPAL against three other state-of-the-art
constrained composite optimization solvers, while the second subsection compares
them against the ${\cal O}(\varepsilon^{-2})$ complexity
method in \cite{ErrorBoundJzhang-ZQLuo2020, ADMMJzhang-ZQLuo2020}. 

We start by describing the details of the IAIPAL variants IPL, IPL(A1), and IPL(A2).  All of them use the parameters
\[
c_{1}=\max\left\{ 1,\frac{L_{f}}{\|{\cal{A}}\|^{2}}\right\} ,\quad\sigma=\frac{1}{\sqrt{2}},\quad\nu=\sqrt{\sigma\left(\lam L_{f}+1\right)},\quad\tau=2.
\]
IPL is as described in Subsection~\ref{Sec:dynamic-IPAAL} with $\alpha = 1/\|A\|^2$, while IPL(A1) and IPL(A2) are a modification of IPL where the ACG subroutine is replaced with an adaptive ACG variant whose specific description can be found in
\cite[Section 5.2]{KongThesis2021}. The difference between the latter ACG variant
compared to the first one is that
the latter one adapts its proximal gradient step to the local curvature of its objective function (see the discussion in the second paragraph following ACG in Appendix~\ref{sec-nesterov}). IPL(A1) chooses $\alpha=1/\|A\|^2$ while IPL(A2) chooses $\alpha=1/c$, i.e., a multiplier update is performed at every outer iteration.

We now describe the other methods used in the first subsection, namely, two variants of the QP-AIPP method of \cite{WJRproxmet1} (nicknamed QP and QP(A)), a variant of the R-QP-AIPP method of \cite{WJRComputQPAIPP} (nicknamed RQP), and the iALM of \cite{ImprovedShrinkingALM20}. 
QP is the method in \cite[Algorithm 4.1.1]{KongThesis2021} while QP(A) is a modification of QP that replaces its ACG subroutine with the same adaptive ACG variant used by IPL(A). RQP is the variant in \cite[Algorithm 5.4.1]{KongThesis2021} which adds another level of adaptability to QP(A) in the sense that its prox parameter $\lam$ is also adapted to the local curvature of the objective function (see the discussion in \cite[Section 1]{WJRComputQPAIPP}). 
Our implementation of iALM uses the parameters
\[
\sigma=2,\quad\beta_{0}=\max\left\{ 1,\frac{L_{f}}{\|{\cal{A}}\|^{2}}\right\} ,\quad w_{0}=1,\quad\boldsymbol{y}^{0}=0,\quad\gamma_{k}=\frac{\left(\log2\right)\|c(x^{1})\|}{(k+1)\left[\log(k+2)\right]^{2}},
\]
for every $k\geq1$. Moreover, the starting point given to the $k$-th
APG call (in the iALM) is set to be $\boldsymbol{x}^{k-1}$, which
is the prox center for the $k$-th prox subproblem. 

We next describe the other methods used in the second subsection, namely, the three variants of the S-prox-ALM of \cite{ErrorBoundJzhang-ZQLuo2020, ADMMJzhang-ZQLuo2020} (nicknamed SPA1--SPA3). The parameter quadruple
$(\alpha,p,c,\beta)$ and initial points $(y_0, z_0)$ used by all the three variants are
\[
\alpha=\frac{\Gamma}{4},\quad p=2(L_{f}+\Gamma\|A\|^{2}),\quad c=\frac{1}{2(L_{f}+\Gamma\|A\|^{2})},\quad\beta=0.5,\quad y_{0}=0,\quad z_{0}=x_{0},
\]
where $\Gamma=0.1$, $1$, and $10$ for SPA1, SPA2, and SPA3, respectively. Note that the choice of $(\alpha,p,c,\beta)$ above with $\Gamma=10$ is the one that is used in the limited quadratic programming experiments of \cite[Section 6.2]{ADMMJzhang-ZQLuo2020}.
Moreover, the aforementioned reference establishes
the iteration-complexity of S-Prox-ALM
for a range of sufficiently small parameters $\beta$  that does not necessarily include the assigned value above, i.e,
$\beta=0.5$.

Some additional technical details about the experiments are also
as follows. First, all of the
tables below report the total number of innermost iterations that each of the methods need to obtain a quadruple satisfying \eqref{eq:comp_crit} below.
This is done so that the iteration
cost for reported in our experiments
are comparable in the sense that each method require ${\cal O}(1)$ resolvent
and gradient evaluations per iteration.
Second, the  algorithms are implemented in MATLAB
2021a and are run on a Windows 64-bit machine with two Intel(R) Xeon(R) Gold 6240 processors and 12 GB of RAM. Third, bold text in the tables of this section indicates the method that performed the most efficiently in a particular metric and problem instance. Finally, the log KKT gap in the experiments below refers to the normalized quantity
\[
\hat{r}:=\log_{10}\left(\max\left\{ \frac{\|\hat{w}\|}{1+\|\nabla f(z_{0})\|},\frac{\|A\hat{z}-b\|}{1+\|Az_{0}-b\|}\right\} \right).
\]

\subsection{Quadratic SDP}

\label{subsec:qsdp}

This subsection presents the performance of IAIPAL method against
several benchmark methods on a set of nonconvex quadratic semidefinite
programming (QSDP) problems. 

Given a pair of dimensions $(\ell,n)\in\mathbb{N}^{2}$, a scalar
pair $(\omega_1,\omega_2)\in\R_{++}^{2}$, linear operators ${\cal {\cal Q}}:\mathbb{S}_{+}^{n}\mapsto\R^{\ell}$,
${\cal {\cal B}}:\mathbb{S}_{+}^{n}\mapsto\R^{n}$, and ${\cal {\cal C}}:\mathbb{S}_{+}^{n}\mapsto\R^{\ell}$
defined by 
\[
\left[{\cal Q}(Z)\right]_{i}=\left\langle Q_{i},Z\right\rangle ,\quad\left[{\cal B}(Z)\right]_{j}=\left\langle B_{j},Z\right\rangle ,\quad\left[{\cal C}(Z)\right]_{i}=\left\langle C_{i},Z\right\rangle ,
\]
for matrices $\{Q_{i}\}_{i=1}^{\ell},\{B_{j}\}_{j=1}^{n},\{C_{i}\}_{i=1}^{\ell}\subseteq\R^{n\times n}$,
positive diagonal matrix $D\in\R^{n\times n}$, and a vector pair
$(b,d)\in\R^{\ell}\times\R^{\ell}$, this subsection considers the
following QSDP: 
\begin{align*}
\min_{Z}\  & \left[f(Z):=-\frac{\omega_{1}}{2}\|D{\cal B}(Z)\|^{2}+\frac{\omega_{2}}{2}\|{\cal C}(Z)-d\|^{2}\right]\\
\text{s.t.}\  & {\cal Q}(Z)=b,\quad Z\in P^{n},
\end{align*}
where $P^{n}=\{Z\in\mathbb{S}_{+}^{n}:{\rm trace}\,(Z)=1\}$. In particular, the problem instances tested are given in Table~\ref{tab:qmp}.

\begin{table}[tbh]
\begin{centering}
\begin{tabular}{>{\raggedright}m{0.4cm}>{\raggedright}m{0.4cm}>{\centering}p{1.1cm}>{\centering}p{1.1cm}>{\centering}p{1.1cm}>{\centering}p{1.1cm}>{\centering}p{1.1cm}>{\centering}p{1.1cm}>{\centering}p{1.1cm}}
\toprule 
\multicolumn{9}{c}{\textbf{\scriptsize{}Log KKT Gap / Log Function Value (row 1) }}\tabularnewline
\multicolumn{9}{c}{\textbf{\scriptsize{}Iteration count / Runtime (row 2)}}\tabularnewline
\midrule
{\tiny{}$m_{f}$} & {\tiny{}$L_{f}$} & {\tiny{}iALM} & {\tiny{}QP} & {\tiny{}QP(A)} & {\tiny{}RQP} & {\tiny{}IPL} & {\tiny{}IPL(A1)} & {\tiny{}IPL(A2)}\tabularnewline
\midrule 
{\tiny{}$10^{0}$} & {\tiny{}$10^{2}$} & {\tiny{}-4.1/-0.98} & {\tiny{}-4.2/-0.98} & {\tiny{}-4.0/-0.98} & {\tiny{}-4.2/-0.98} & {\tiny{}-4.0/-0.98} & {\tiny{}-4.0/-0.98} & {\tiny{}-4.0/-0.98}\tabularnewline
 &  & {\tiny{}20.4/10.3} & {\tiny{}4.1/3.8} & {\tiny{}3.2/3.5} & {\tiny{}1.9/2.1} & {\tiny{}3.0/3.1} & \textbf{\tiny{}1.4/1.7} & {\tiny{}1.4/1.7}\tabularnewline
\midrule
{\tiny{}$10^{0}$} & {\tiny{}$10^{3}$} & {\tiny{}-4.1/0.04} & {\tiny{}-4.2/0.04} & {\tiny{}-4.3/0.04} & {\tiny{}-4.2/0.04} & {\tiny{}-4.0/0.04} & {\tiny{}-4.0/0.04} & {\tiny{}-4.0/0.04}\tabularnewline
 &  & {\tiny{}36.2/18.6} & {\tiny{}3.9/3.7} & {\tiny{}4.4/4.7} & {\tiny{}2.0/2.2} & {\tiny{}1.7/1.7} & \textbf{\tiny{}0.8/0.9} & {\tiny{}0.8/0.9}\tabularnewline
\midrule
{\tiny{}$10^{0}$} & {\tiny{}$10^{4}$} & {\tiny{}-4.3/1.04} & {\tiny{}-4.3/1.04} & {\tiny{}-4.3/1.04} & {\tiny{}-4.3/1.04} & {\tiny{}-4.3/1.04} & {\tiny{}-4.3/1.04} & {\tiny{}-4.3/1.04}\tabularnewline
 &  & {\tiny{}102.1/51.5} & {\tiny{}4.1/3.7} & {\tiny{}9.3/10.1} & {\tiny{}2.5/2.7} & {\tiny{}1.3/1.3} & \textbf{\tiny{}0.6/0.8} & {\tiny{}0.6/0.8}\tabularnewline
\midrule
{\tiny{}$10^{1}$} & {\tiny{}$10^{5}$} & {\tiny{}-4.3/2.04} & {\tiny{}-4.3/2.04} & {\tiny{}-4.3/2.04} & {\tiny{}-4.3/2.04} & {\tiny{}-4.3/2.04} & {\tiny{}-4.3/2.04} & {\tiny{}-4.3/2.04}\tabularnewline
 &  & {\tiny{}104.3/52.5} & {\tiny{}4.1/3.8} & {\tiny{}9.3/10.0} & {\tiny{}2.5/2.8} & {\tiny{}3.3/3.4} & {\tiny{}1.5/1.8} & \textbf{\tiny{}0.6/0.8}\tabularnewline
\midrule 
{\tiny{}$10^{2}$} & {\tiny{}$10^{5}$} & {\tiny{}-4.3/2.04} & {\tiny{}-4.2/2.04} & {\tiny{}-4.3/2.04} & {\tiny{}-4.2/2.04} & {\tiny{}-4.1/2.04} & {\tiny{}-4.1/2.04} & {\tiny{}-4.0/2.04}\tabularnewline
 &  & {\tiny{}48.9/24.7} & {\tiny{}3.9/3.6} & {\tiny{}4.4/4.8} & {\tiny{}2.0/2.2} & {\tiny{}2.4/2.4} & {\tiny{}1.0/1.3} & \textbf{\tiny{}0.8/0.9}\tabularnewline
\midrule 
{\tiny{}$10^{3}$} & {\tiny{}$10^{5}$} & {\tiny{}-4.2/2.02} & {\tiny{}-4.1/2.02} & {\tiny{}-4.0/2.02} & {\tiny{}-4.2/2.02} & {\tiny{}-4.0/2.02} & {\tiny{}-4.0/2.02} & {\tiny{}-4.0/2.02}\tabularnewline
 &  & {\tiny{}39.8/20.3} & {\tiny{}4.4/4.1} & {\tiny{}3.7/3.9} & {\tiny{}2.1/2.4} & {\tiny{}2.0/2.1} & \textbf{\tiny{}0.9/1.1} & {\tiny{}0.9/1.1}\tabularnewline
\bottomrule
\end{tabular}
\par\end{centering}
\centering{}\caption{Results for constant $m_{f}$ and variable $L_{f}$. Within each $(m_{f},L_{f})$
multirow, the first row presents $\log_{10}$ KKT gaps / $\log_{10}$
function values, while the second row presents iteration counts (in
thousands) / runtimes (in tens of seconds). Log function value is
computed as $\log_{10}\phi(\hat{z})$.\label{tab:qmp}}
\end{table}

We now describe the experiment parameters. First, the dimensions are $(\ell,n)=(30,100)$
and only 5\% of the entries of $Q_{i},B_{j},$ and $C_{i}$ are nonzero.
Second, the entries of $Q_{i}$, $B_{j}$, $C_{i}$, $D$, $b$, and
$d$ are generated using the procedure described in \cite[Subsection 5.5.2.1]{KongThesis2021}.
Fifth, given a starting point $z_{0}\in\r^{n\times n}$, all of the
methods attempt to find a quadruple $(\hat{z},\hat{p},\hat{w},\hat{q})$
satisfying $\hat{w}\in\nabla f(\hat{z})+\partial\delta_{P^{n}}(\hat{z})+{\cal Q}^{*}\hat{p}$ and 
\begin{gather}
\frac{\|\hat{w}\|}{1+\|\nabla f(z_{0})\|}\leq \hat\rho ,\quad\frac{\|{\cal{Q}}\hat{z}-b\|}{1+\|{\cal{Q}}z_{0}-b\|}\leq\hat\eta,
\label{eq:comp_crit}
\end{gather}
with $\hat\rho = \hat\eta = 10^{-4}$.
Sixth, using the fact that $\|Z\|_{F}\leq1$ for every $Z\in P_{n}$,
the constant hyperparameters for the IPL and iALM methods are set
to $L_{g}=0$, $L_{j}=0$, $\rho_{j}=0$, and $B_{j}=\|Q_{j}\|_{F}$
for $1\leq j\leq\ell$. Finally, each problem instance considered
is based on a specific pair $(m_{f},L_{f})$ for which the scalar
pair $(\omega_{1},\omega_{2})$ is selected so that $L_{f}=\lam_{\max}(\nabla^{2}f)$
and $m_{f}=-\lam_{\min}(\nabla^{2}f)$.

We now make several observations and conclusions based on these tables.
First, comparing
the results between QP(A)
and
RQP, we conclude that the presence of
an adaptive prox stepsize search in the latter method considerably improves
its performance compared
to the former.
Second, IPL(A2) is the direct counterpart of QP(A), but
its performance is
better than
the improved version of
QP(A), namely
RQP, in nine out of ten problem instances.
Third, in view of the
first remark above,
it is reasonable to infer that
IPL(A1) and IPL(A2) could be considerably
improved if the prox parameter $\lam$ is adaptively chosen. As the analysis
for such an IPL variant involves several technical difficulties, we leave its development for a future work.

\subsection{Comparison with an  \texorpdfstring{${\cal O}(\varepsilon^{-2})$}{O(ε^{-2}} complexity method}

\label{subsec:numerical_cmp}


We start by comparing and contrasting the theoretical properties of
each method. First, both the IAIPAL method and the S-prox-ALM are augmented
Lagrangian-based methods applied to NCO problems.
More specifically, SPA considers \eqref{optl0} under the requirement
that $h$ is the indicator function of a polyhedron. Second, the S-prox-ALM
also considers a sequence of proximal subproblems as in \eqref{exactzk0}, and applies a single composite gradient step to inexact solve \eqref{exactzk0} instead of an ACG-type subroutine. Finally, while the IAIPAL method only requires choosing its parameters based
on the scalars $m_{f}$, $L_{f}$, and $\|A\|$ to guarantee convergence, the S-prox-ALM requires
choosing its parameters based on the supremum of a set of Hoffman
constants (see the proof of \cite[Lemma 3.10]{ADMMJzhang-ZQLuo2020} and \cite[Lemma 4.8]{ADMMJzhang-ZQLuo2020})
that is generally difficult to compute.

We now present some numerical results that compare the S-prox-ALM variants against IP(A1), QP(A), and RQP. Since the S-prox-ALM does not have convergence guarantees for the QSDP problem
in Subsection~\ref{subsec:qsdp} (because the domain of $h$ is not
polyhedral), we consider the vector variant of the QSDP. More specifically,
given a pair of dimensions $(\ell,n)\in\mathbb{N}^{2}$, a scalar
pair $(\omega_1, \omega_2)\in\R_{++}^{2}$, matrices $Q,C\in\R^{\ell\times n}$
and $B\in\r^{n\times n}$, positive diagonal matrix $D\in\R^{n\times n}$,
and a vector pair $(b,d)\in\R^{\ell}\times\R^{\ell}$, we consider
the problem 
\begin{align*}
\min_{z}\  & \left[f(z)-\frac{\omega_{1}}{2}\|DBz\|^{2}+\frac{\omega_{2}}{2}\|{\cal C}z-d\|^{2}\right]\\
\text{s.t.}\  & Qz=b,\quad z\in\Delta^{n},
\end{align*}
where $\Delta^{n}:=\{x\in\rn:\sum_{i=1}^{n}x_{i}=1\}$. In particular, the problem instances tested are given in Table~\ref{tab:qvp}.

\begin{table}[tbh]
\begin{centering}
\begin{tabular}{>{\raggedright}m{0.4cm}>{\raggedright}m{0.4cm}>{\centering}p{1.1cm}>{\centering}p{1.1cm}>{\centering}p{1.1cm}>{\centering}p{1.1cm}>{\centering}p{1.1cm}>{\centering}p{1.1cm}}
\toprule 
\multicolumn{8}{c}{\textbf{\scriptsize{}Log KKT Gap / Log Function Value (row 1)}}\tabularnewline
\multicolumn{8}{c}{\textbf{\scriptsize{}Iteration count (row 2)}}\tabularnewline
\midrule
{\tiny{}$m_{f}$} & {\tiny{}$L_{f}$} & {\tiny{}QP(A)} & {\tiny{}RQP} & {\tiny{}IPL(A1)} & {\tiny{}SPA1} & {\tiny{}SPA2} & {\tiny{}SPA3}\tabularnewline
\midrule 
{\tiny{}$10^{0}$} & {\tiny{}$10^{2}$} & {\tiny{}-2.0/2.30} & {\tiny{}-2.9/2.30} & \textbf{\tiny{}-3.2}{\tiny{}/2.30} & {\tiny{}-2.3/2.30} & {\tiny{}-1.7/2.30} & {\tiny{}-1.1/2.30}\tabularnewline
 &  & {\tiny{}11.5} & {\tiny{}12.2} & {\tiny{}11.6} & {\tiny{}9.9} & {\tiny{}9.8} & {\tiny{}9.9}\tabularnewline
\midrule
{\tiny{}$10^{0}$} & {\tiny{}$10^{3}$} & {\tiny{}-1.2/2.30} & {\tiny{}-2.6/2.30} & \textbf{\tiny{}-3.2}{\tiny{}/2.30} & {\tiny{}-1.9/2.30} & {\tiny{}-1.7/2.30} & {\tiny{}-1.1/2.30}\tabularnewline
 &  & {\tiny{}11.7} & {\tiny{}12.2} & {\tiny{}11.2} & {\tiny{}8.8} & {\tiny{}9.8} & {\tiny{}9.8}\tabularnewline
\midrule
{\tiny{}$10^{0}$} & {\tiny{}$10^{4}$} & {\tiny{}-0.7/2.30} & {\tiny{}-2.3/2.30} & \textbf{\tiny{}-4.1}{\tiny{}/2.30} & {\tiny{}-1.7/2.30} & {\tiny{}-1.7/2.30} & {\tiny{}-1.5/2.30}\tabularnewline
 &  & {\tiny{}11.9} & {\tiny{}12.1} & {\tiny{}13.6} & {\tiny{}9.9} & {\tiny{}9.9} & {\tiny{}9.9}\tabularnewline
\midrule
{\tiny{}$10^{1}$} & {\tiny{}$10^{5}$} & {\tiny{}-0.7/2.39} & {\tiny{}-2.3/2.40} & \textbf{\tiny{}-4.1}{\tiny{}/2.41} & {\tiny{}-0.7/2.39} & {\tiny{}-1.4/2.41} & {\tiny{}-1.7/2.41}\tabularnewline
 &  & {\tiny{}11.9} & {\tiny{}12.2} & {\tiny{}12.2} & {\tiny{}9.9} & {\tiny{}9.9} & {\tiny{}9.8}\tabularnewline
\midrule 
{\tiny{}$10^{2}$} & {\tiny{}$10^{5}$} & {\tiny{}-0.9/2.35} & {\tiny{}-2.5/2.37} & \textbf{\tiny{}-3.4}{\tiny{}/2.37} & {\tiny{}-0.7/2.35} & {\tiny{}-1.4/2.37} & {\tiny{}-1.8/2.37}\tabularnewline
 &  & {\tiny{}10.2} & {\tiny{}12.2} & {\tiny{}8.5} & {\tiny{}9.9} & {\tiny{}9.9} & {\tiny{}9.9}\tabularnewline
\midrule 
{\tiny{}$10^{3}$} & {\tiny{}$10^{5}$} & {\tiny{}-1.5/1.81} & {\tiny{}-2.4/1.83} & \textbf{\tiny{}-7.0}{\tiny{}/1.83} & {\tiny{}-0.6/1.75} & {\tiny{}-1.3/1.87} & {\tiny{}-2.0/1.87}\tabularnewline
 &  & {\tiny{}9.5} & {\tiny{}12.0} & {\tiny{}3.0} & {\tiny{}9.9} & {\tiny{}9.9} & {\tiny{}9.7}\tabularnewline
\bottomrule
\end{tabular}
\par\end{centering}
\centering{}\caption{Results for constant $m_{f}$ and variable $L_{f}$. Within each $(m_{f},L_{f})$
multirow, the first row presents $\log_{10}$ KKT gaps / $\log_{10}$
function values, while the second row presents iteration counts (in
thousands). All experiments were run until 600 seconds have passed.
Log function value is computed as $\log_{10}(200+\phi(\hat{z}))$.\label{tab:qvp}}
\end{table}

We now describe the experiment parameters for the problem instances
considered. First, the dimension pair is $(\ell,n)=(20,1000)$ and
all generated matrices  have full density. Second, the entries
of $Q$, $B$, $C$, and $d$ (resp. $D$) are generated by sampling
from the uniform distribution ${\cal U}[0,1]$ (resp. ${\cal U}\{1,...,1000\}$).
Third, the vector $b$ is set to $b=Q(\boldsymbol{e}/n)$ where $\boldsymbol{e}$
is a vector of all ones. Fourth, the initial starting point $z_{0}$
is a set to be $\tilde{z}/\sum_{i=1}^{n}\tilde{z}_{i}$ where the
entries of $\tilde{z}$ are sampled from the ${\cal U}[0,1]$ distribution.
Fifth, given a starting point $z_{0}\in\r^{n}$, all of the methods
attempt to find a quadruple $(\hat{z},\hat{p},\hat{w},\hat{q})$ satisfying $\hat{w}\in\nabla f(\hat{z})+\partial\delta_{\Delta^{n}}(\hat{z})+Q^{*}\hat{p}$ and \eqref{eq:comp_crit} with $\hat\rho=\hat\eta=10^{-7}$. 
Finally, all experiments are run with a time limit of 600 seconds.

From the results, we
can see that the IPL(A2) variant is substantially more efficient than SPA1--SPA3, QP(A), and RQP. We also notice that SPA1 (resp. SPA3)
tends to perform better when $L_{f}$ is small (large).

\section{Concluding remarks}\label{sec:conclusion}

This paper proposes the IAIPAL method
for finding a $(\hat\rho,\hat\eta)$-approximate stationary point (see Definition~\ref{def:stationarypoint}) of a class of
linearly-constrained smooth NCO problems,
and establishes, up to logarithmic terms, an
$\mathcal{O}(\hat\rho^{-5/2}+\hat\rho^{-2}\hat\eta^{-1/2})$  
ACG iteration complexity bound for it.
Moreover, IAIPAL is the first PAL method with provable complexity bounds for the case where $(\theta, \chi_k)=(0,1)$ in \eqref{eq:gen_dual_update2} for every $k\geq 1$.
Computational results also show that
IAIPAL substantially outperforms other algorithms in the literature for solving \eqref{optl0} (or special cases of it).

We now discuss some possible extensions of our paper. First, it is worth developing an adaptive variant of IAIPAL as described in the conclusion of Section~\ref{sec:numerical}. Second, one could analyze the convergence and computational behavior of IAIPAL under the multiplier update rule $p_{k+1}=p_k+ \chi c(A z_k-b)$ where $\chi$ is a positive scalar lying in a certain range. Finally, it is worth investigating whether the iteration-complexity of IAIPAL can be improved, possibly for special instances of \eqref{optl0}.

	\appendix
	\section{Other technical results}
This section is divided into three subsections. The first one revise an accelerated gradient method used for solving the IAIPAL subproblems. The second subsection establishes a result, using convex analysis, that is used to prove Lemma~\ref{lem:bound_xiN}. The last subsection presents a result regarding a refinement procedure related to the pair $(\hat z,\hat{w})$ computed in step~2 of S-IAIPAL.

		\subsection{An accelerated composite gradient method}\label{sec-nesterov}
	
	
    Consider the following composite optimization problem
	\begin{equation}\label{mainprob:nesterov1}
	\min \{\psi(x):=\psi_s(x)+\psi_n(x) :  x \in \Re^n\}
	\end{equation}
	where the following conditions are assumed to hold:
	\begin{itemize}
		\item [{\bf(A1)}]$\psi_n:\Re^n\rightarrow (-\infty,+\infty]$ is a proper closed  convex  function;
		\item [{\bf (A2)}]$\psi_s$ is a convex differentiable function on $\dom \psi_n$
		and there exists  $({\widetilde \mu},{\widetilde M})\in \r_{+}^2$ satisfying $\widetilde M > \widetilde \mu$ and
		${\widetilde \mu}\|u-x\|^2/2 \leq \psi_s(u)-\ell_{\psi_s}(u; x) 
		\leq  {\widetilde M}\|u-x\|^2/2$  for every $x, u \in \dom \psi_n$, where $\ell_{\psi_s}(\cdot \,;\cdot)$ is defined in \eqref{eq:defell}.
	\end{itemize}


	We are now ready to state ACG. It is worth mentioning that
    other ACG variants such as the ones
	in \cite{Attouch2016,YHe2,nesterov2012gradient,nesterov1983method}
	could also be used in the development of
	 IAIPAL.

\vgap

\hrule
\vspace*{0.5em}
\noindent\quad\textbf{ACG}
\vspace*{0.5em}
\hrule
\vspace*{0.5em}
\begin{itemize}
\item[(0)] Let a pair of functions $(\psi_{s},\psi_{n})$ satisfying \textbf{(A1)
}and \textbf{(A2)} for some $({\widetilde \mu},\widetilde M)\in\r_{+}^{2}$, a scalar
$\tilde{\sigma}>0$, and an initial point $y_{0}\in\dom\psi_{n}$
be given; set $x_{0}=y_{0}$, $A_{0}=0$, $\tau_0=1$, $\zeta = 1/({\widetilde M} - {\widetilde \mu})$, and $j=0$;
\item[(1)] compute the iterates 
\begin{align*}
 a_{j} & =\frac{\zeta\tau_{j}+\sqrt{(\zeta\tau_{j})^{2}+4\tau_{j}A_{j}}}{2}, \quad A_{j+1}=A_{j}+a_{j}, \quad \tilde{x}_{j} =\frac{A_{j} y_j + a_j x_j}{A_{j+1}}\\
\tau_{j+1} & =\tau_{j} +{\widetilde \mu} a_{j}, \quad y_{j+1} =\argmin_{y\in\rn}\left\{ \ell_{\psi_s}(y;\tilde{x}_j) + \psi_n(y) +\frac{{\widetilde M}}{2}\|y-\tilde{x}_{j}\|^{2}\right\}, \\
x_{j+1} & = \frac{1}{\tau_{j+1}} \left[\frac{a_j}{\zeta}(y_{j+1} - \tilde{x}_j) + {\widetilde \mu} a_j y_{j+1} + \tau_j x_j \right];
\end{align*}
\item[(2)] compute the quantities
\begin{align*}
 u_{j+1} &= {\widetilde \mu}(y_{j+1} - x_{j+1}) + \frac{x_{0}-x_{j+1}}{A_{j+1}}, \\
 \eta_{j+1} &= \frac{1}{2 A_{j+1}}\left(\|x_0 - y_{j+1}\|^2 - \tau_{j+1}\|x_{j+1}-y_{j+1}\|^2\right);
\end{align*}
\item[(3)] if the inequality 
\[
\|u_{j+1}\|^{2}+2\eta_{j+1}\leq\tilde{\sigma}^{2}\|y_{0}-y_{j+1}+u_{j+1}\|^{2}
\]
holds, then stop and output $(y,u,\eta):=(y_{j+1},u_{j+1},\eta_{j+1})$;
otherwise, set $j=j+1$ and go to (1).
\end{itemize}
\vspace*{0.5em}
\hrule
\vgap
	
 	Some remarks about ACG follow. First, the most common way of describing an iteration of ACG is as in step~1. 
	Second, the
	auxiliary iterates $\{u_j\}$ and $\{\eta_j\}$ computed  in step~2 are  used to develop a stopping criterion for ACG  when it is called as a subroutine for solving the subproblems generated in step~1 of S-IAIPAL
	in Subsection~\ref{Sec:static-IPAAL}.
Third, it can be shown (see for example   \cite{florea2018accelerated, fistaReport2021}) that ACG (without steps 2 and 3) with ${\widetilde \mu}=0$ corresponds to the well-known FISTA algorithm. 
	Fourth, the sequence $\{A_j\}$ has the following increasing property:
	$$ A_{j}\geq\frac{1}{{\widetilde M-\widetilde \mu}}\max\left\{\frac{j^{2}}{4},\left(1+\sqrt{\frac{{\widetilde \mu}}{4(\widetilde M-\widetilde \mu)}}\right)^{2(j-1)}\right\}, \qquad \forall j\geq 1. 
	$$
Finally, it is worth mentioning that adaptive variants\footnote{The closest variant to ACG in this paper can be found in \cite[Section 5.2]{KongThesis2021}.} of ACG have been studied, for example, in \cite{beck2009fast, KongThesis2021, pmlr-v32-lin14, nesterov2012gradient, parikh2014proximal}. A simple level of adaptiveness
used in these variants, which is also used
inside some of the methods benchmarked in Section~4,
is to replace ${\widetilde M}$ in
the computation of $y_j$ in step~1 by
an estimate $M_j$ computed as follows:
$M_j$ is initially set to be $M_{j-1}$ and, if necessary, is repeatedly increased (either additively, multiplicatively, or both) until the inequality $\psi_s(y_j) - \ell_{\psi_s}(y_j;\tilde x_{j-1})\leq M_{j} \|y_j - \tilde{x}_{j-1}\|^2 / 2 $ is satisfied.

The next result, whose proof can be found in \cite[Lemma~2.13]{fistaReport2021}, summarizes  the main properties of ACG used in this paper.

	\begin{proposition}\label{prop:nest_complex}
		Let $\{(y_j, u_j,\eta_j)\}_{j\geq 1}$ be the sequence generated by ACG applied to \eqref{mainprob:nesterov1},
		where $(\psi_s,\psi_n)$ is a given pair of data functions satisfying {\bf (A1)} and {\bf (A2)}. 
		Then, the following statements hold: 
		\begin{itemize}
			\item [a)] for every $j\geq 1$, we have 
			$u_j\in  \partial_{\eta_j}(\psi_s+\psi_n)(y_j)$;
			\item[b)] for any $\tilde \sigma>0$, the ACG method outputs a triple $(y,u,\eta)$ 
			  satisfying  
			$$
			u\in  \partial_{\eta}(\psi_s+\psi_n)(y) \quad \|u\|^{2}+2\eta\le \tilde\sigma^2\|y_{0}-y+u\|^{2}
			$$
			in at most 
			\begin{equation}
\left\lceil\left(\frac{1}{2}+\sqrt{\frac{\widetilde M-\widetilde\mu}{\widetilde\mu}}\right)\log_{1}^{+}\left(\left[\widetilde M-\widetilde\mu\right]{\cal A}_{\widetilde\mu,\widetilde\sigma}\right)+1 \right\rceil
\end{equation}
iterations, where
${\cal A}_{\widetilde\mu,\widetilde\sigma}:=(2\widetilde\mu+3)(1+\widetilde\sigma)^{2}/\widetilde\sigma^2$.
		\end{itemize}
	\end{proposition}

\subsection{A convex analysis result}
This subsection  contains  a technical result of convex analysis. It  derives several characterizations of
	condition {\bf (B2)} and establishes  an important
	inclusion that is used in the
	proof of Lemma~\ref{lem:bound_xiN}.
	
	\begin{lemma}\label{Prop:eps-subd-bounded-cone}
		Let $h \in \bConv{n}$ and $L_h \ge 0$ be given.
		Then, the following statements are equivalent:
		\begin{itemize}
			\item[a)]
			for every $z,z' \in {\cal H}$, we have
			$
			h(z') \le h(z) + L_h \|z'-z\|;
			$
			\item[b)]
			for every $z,z' \in {\cal H}$, we have
			$
			h'(z;z'-z) \le L_h \|z'-z\|;
			$
 			\item[c)]
			for every $z,z' \in {\cal H}$ and $s \in \partial h(z)$, we have
			$
			\inner{s}{z'-z} \le L_h \|z'-z\|;
			$
			\item[d)]
			for every $z \in {\cal H}$, we have
			$
			\partial h(z) \subset \bar B(0;L_h) + N_{\cal H}(z);
			$
			\item[e)]
			for every $z \in {\cal H}$, we have
			$\partial h(z) \cap \bar B(0;L_h) \ne \emptyset.$
		\end{itemize}
		Moreover, any of the above conditions imply
		that:
		\begin{itemize}
		    \item[i)] ${\cal H}$ is closed;
		    \item[ii)]
		    for any $z \in {\cal H}$ and $\varepsilon \ge 0$, we have
		    $
		\partial_{\varepsilon} h(z) \subset  \bar B(0;L_h)
		+ N_{\cal H}^{\varepsilon}(z).
		$
		\end{itemize}
	\end{lemma}
	
	\begin{proof}
	\noindent	[$a) \Rightarrow b)$] This statement follows from the fact that
		$h(z')-h(z) \ge h'(z;z'-z)$ for every $z,z' \in {\cal H}$
		(see  \cite[Theorem~23.1]{Rockafellar70}).
		
	\noindent	[$b) \Rightarrow c)$] This statement follows from the fact that
		$h'(z;z'-z)\geq \inner{s}{z'-z}$ for every $z,z' \in {\cal H}$ and $s\in \partial h(z)$,
		(see  \cite[Theorem~23.2]{Rockafellar70}).

	\noindent	[$c) \Rightarrow d)$]
		Letting $T_{\cal H}(z) = \cl( \R_+ \cdot ({\cal H}-z) )$ and $N_{\cal H}(z)$
		denote the tangent
		cone and normal cone of ${\cal H}$ at $z$, respectively,
		and letting $S:= \bar B(0;L_h) + N_{\cal H}(z)$,
		we easily see that c) is equivalent to
		\[
		\inner{s}{\cdot}  \le  L_h\|\cdot\| + I_{T_{\cal H}(z)} (\cdot) 
		= \sigma_{\bar B(0;L_h)}(\cdot) + \sigma_{N_{\cal H}(z)}(\cdot) = \sigma_S(\cdot) \quad \forall s \in \partial h(z),
		\]
		where  the first equality follows in view of the discussion in  page 115 of \cite{Rockafellar70} and   \cite[Example 2.3.1 combined with Proposition 5.2.4]{Hiriart1}, the last equality is due to  \cite[Corollary 16.4.1]{Rockafellar70}.	
		Since the above hold for every $s \in \partial h(z)$, we conclude that
		$\sigma_{\partial h(z)} \leq \sigma_{S}$.
		Since both $\partial h(z)$ and $S$ are closed,
		it follows from \cite[Corollary 13.1.1]{Rockafellar70}  that
		$\partial h(z) \subset S = \bar B(0;L_h) + N_{\cal H}(z)$.
		
		\noindent [$d) \Rightarrow e)$]
		Assume that d) holds. We will first show that
		e) holds for every $z \in \ri {\cal H}$. Indeed,
		assume that $z \in \ri {\cal H}$. This implies that
		$N_{\cal H}(z)$ is a subspace, namely, the one orthogonal
		to the subspace parallel to the affine hull of ${\cal H}$.
		It follows from d) that there exists
		$s \in \partial h(z)$ and $n \in N_{\cal H}(z)$
		such that $\|s-n\| \le L_h$. Since $N_{\cal H}(z)$
		is a subspace, it follows that $-n \in N_{\cal H}(z)$.
		The claim now follows by the observation
		that $s \in \partial f(z)$ and $-n \in N_{\cal H}(z)$
		immediately implies that $s-n \in \partial f(z)$.
		We will now show that e) also holds for every
		$z \in \rbd {\cal H}$. Indeed, assume that
		$z \in \rbd {\cal H}$. Then, due to \cite[Proposition 2.1.8]{Hiriart1}, there exists
		$\{z_k\} \subset \ri {\cal H}$ such that $z_k$ converges
		to $z$ as $k \to \infty$.
		Since e) holds for every $z \in \ri {\cal H}$
		and $\{z_k\} \subset \ri {\cal H}$, we conclude that
		for every $k$,
		there exists $s_k \in \partial h(z_k)$ such that
		$\|s_k\| \le L_h$. Hence, by Bolzano-Weisstrass'
		theorem, there exists a subsequence
		$\{s_k\}_{k\in {\cal K}}$ converging to some
		$s$, which clearly satisfies $\|s\| \le L_h$.
		Using the fact that $\{(z_k,s_k)\}_{k\in {\cal K}} \in \Gr (\partial h)$ and $\{(z_k,s_k)\}_{k\in {\cal K}}$
		converges to $(z,s)$, and the fact that
		$h \in \bConv{n}$ implies that the set
		$\Gr (\partial h)$ is closed, we then conclude that $(z,s) \in \Gr (\partial h)$, i.e.,
		$s \in \partial h(z)$. We have thus shown that
		e) holds for every $z \in \rbd{\cal H}$ as well.
		
\noindent		[$e) \Rightarrow a)$]
		Let $z,z' \in {\cal H}$ be given and assume that e) holds.
		Then,
		there exists $s' \in \partial h(z')$ such that
		$\|s' \| \le L_h$.
		Hence, 
		$ h(z) - h(z') \ge \inner{s'}{z-z'}
		\ge  - \|s'\| \, \|z'-z\| \ge -L_h \|z'-z\|,
		$ which proves~ a).   

    \noindent    [$a) \Rightarrow i)$] Assume that
        $\{z_k\}\subset {\cal H} $ converges to
        $z$.
        The fact that $h \in \bConv{n}$ and the assumption that (a) holds
        imply that
        \[
        h(z)\leq \liminf_{k\rightarrow +\infty} h(z_k) \leq
        \liminf_{k\rightarrow +\infty} 
        \left(h(z_1)+L_h\|z_k-z_1\|\right)=
        h(z_1)+L_h\|z-z_1\|<+\infty,
        \]
        and hence that
        $z \in {\cal H}$.
        We have thus shown that ${\cal H}$ is closed.

		\noindent[$a) \Rightarrow ii)$] Let $z \in {\cal H}$ and $\varepsilon \ge 0$ be given
		and assume that a) holds.
		Consider
		the function $\phi_z$ defined as
		\[
		 \phi_z(z') := h(z) + L_h\|z'-z\| + I_{\cal H}(z') \quad \forall z' \in \Re^n.
		 \]
		 Clearly, $\phi_z(z) = h(z)$ and $\phi_z \ge h$
		 in view of a).
		 Using these two observations and the
		 definition of the
		 $\varepsilon$-subdifferential given in \eqref{def:epsSubdiff}, we immediately see that
		 $\partial_{\varepsilon} h(z) \subset \partial _{\varepsilon} \phi_z(z)$.
		 On the other hand, using
		 the $\varepsilon$-subdifferential rule for the sum of two convex functions
		 (see \cite[Theorem 3.1.1]{Hiriart2}),
		 we have that
		 \[
		 \partial_{\varepsilon} \phi_z(z)
		 \subset
		 \partial_{\varepsilon} \left(L_h\|\cdot-z\| \right) (z) + \partial_{\varepsilon} I_{\cal H} (z)
		 =  \partial_{\varepsilon} \left(L_h\|\cdot\| \right)(0) + N^{\varepsilon}_{\cal H} (z),
		 \]
		 where the last equality is due to the 
		 the affine  composition rule
		 for the $\varepsilon$-subdifferential 
		 (see \cite[Theorem 3.2.1]{Hiriart2}) and
		 the fact that
		 $N_{\cal H}^{\varepsilon}(\cdot) = \partial_{\varepsilon} I_{\cal H}(\cdot)$.
The implication now follows from the above two inclusions and the fact that
$\partial_{\varepsilon} \left(L_h\|\cdot\| \right)(0) = \bar B(0;L_h)$.
	\end{proof}
	
	We observe that a) of Lemma~\ref{Prop:eps-subd-bounded-cone}
	is the same as condition {\bf (B2)}.
	Conditions b) to e) are all equivalent to a), and hence
	{\bf (B2)}. The implication a)  $\Rightarrow$  ii)
	is the one that is used in
	the proof of  Lemma~\ref{lem:bound_xiN}.

	

\subsection{A basic refinement result}

Even though the result below,
which is used to prove Proposition~\ref{prop:approxsol},
is a slight variant of \cite[Lemma~32]{YheMoneiroNash},
we include its proof for the sake of completeness.

\begin{lemma}\label{lem:approxsolreps}
Assume that $\tilde h\in\bConv{n}$,
$\tilde g$ is  a differentiable function on ${\dom \tilde h}$,
and $(z,\varepsilon) \in {\dom \tilde h} \times \Re_+$ is such that
\begin{equation}\label{Inc:repsdif}
0 \in \partial_{\varepsilon}( \tilde g + \tilde h) (z).
\end{equation}
Assume also that
there exists $\tilde L>0$ such that  \begin{equation}\label{uppercurvature_hatg}
\tilde g(u)- \ell_{\tilde g}(u; z) \leq\frac{\tilde L}{2}\|u-z\|^2\qquad \forall u \in {\mathcal \dom \tilde h},
\end{equation}
and define
\begin{equation}\label{eq:def_zhat}
{\tilde z} := \argmin_u \left\{ \ell_{\tilde g}(u;z)  + \tilde h(u) + \frac{\tilde L}{2} \|u-z\|^2  \right \},  \qquad \widetilde w:= \tilde L(z-\tilde z).
\end{equation}
Then,  the quadruple $(z,\tilde z, \widetilde w, \varepsilon)$ satisfies 
\begin{equation}\label{eq:inclusion_w}
 \widetilde w \in \nabla \tilde g(z) + \partial \tilde h({\tilde z}), \qquad \widetilde w \in \nabla \tilde g(z) + \partial_{{\varepsilon}} \tilde h(z), \qquad \|\widetilde w\| \leq \sqrt{2 \tilde L\varepsilon}. 
\end{equation}
\end{lemma}

\begin{proof}
 The first inclusion in \eqref{eq:inclusion_w} follows from
 the definition of $\widetilde w$ and
 the optimality condition for the problem in  \eqref{eq:def_zhat}.
Now, using the first inclusion in \eqref{eq:inclusion_w}, the definition of $\widetilde w$ in \eqref{eq:def_zhat}, inclusion \eqref{Inc:repsdif},
inequality \eqref{uppercurvature_hatg},
and the subdifferential definition \eqref{def:epsSubdiff},
we conclude that
for every $u\in \Re^n$,
\begin{align*}
h(u)&\geq h(\tilde z)+\langle \widetilde w-\nabla \tilde  g(z),u-\tilde z\rangle \nonumber\\
&=h(z)+\langle \widetilde w-\nabla \tilde  g(z),u-z\rangle+h(\tilde z)-h(z)+\frac{\|\widetilde w\|^2}{\tilde L}+\langle\nabla \tilde  g(z),\tilde z-z\rangle\\
&\geq h(z)+\langle \widetilde w-\nabla \tilde  g(z),u-z\rangle+h(\tilde z)-h(z)+\frac{\|\widetilde w\|^2}{\tilde L}+g(\tilde z)-g(z) -\frac{\tilde L}{2}\|\tilde z-z\|^2\\
&\geq h(z)+\langle \widetilde w-\nabla \tilde  g(z),u-z\rangle-\varepsilon +\frac{\|\widetilde w\|^2}{2 \tilde L}
\end{align*}
which, in view of \eqref{def:epsSubdiff},
clearly implies the second inclusion in \eqref{eq:inclusion_w}.
Finally, the inequality in \eqref{eq:inclusion_w} follows from  the
above relations with $u=z$.
\end{proof}

\bibliographystyle{plain}
\bibliography{Proxacc_ref}

\def\cprime{$'$}
\begin{thebibliography}{10}

\bibitem{Attouch2016}
H.~Attouch and J.~Peypouquet.
\newblock The rate of convergence of {N}esterov{\textquotesingle}s accelerated
  forward-backward method is actually faster than $1/k^{2}$.
\newblock {\em SIAM J. Optim.}, 26(3):1824--1834, 2016.

\bibitem{Aybatpenalty}
N.S. Aybat and G.~Iyengar.
\newblock A first-order smoothed penalty method for compressed sensing.
\newblock {\em SIAM J. Optim.}, 21(1):287--313, 2011.

\bibitem{AybatAugLag}
N.S. Aybat and G.~Iyengar.
\newblock A first-order augmented {{L}agrangian} method for compressed sensing.
\newblock {\em SIAM J. Optim.}, 22(2):429--459, 2012.

\bibitem{beck2009fast}
A.~Beck and M.~Teboulle.
\newblock A fast iterative shrinkage-thresholding algorithm for linear inverse
  problems.
\newblock {\em SIAM J. Imaging Sci.}, 2(1):183--202, 2009.

\bibitem{Ber1}
D.~P. Bertsekas.
\newblock {\em Constrained optimization and Lagrange multiplier methods}.
\newblock Academic Press, New York, 1982.

\bibitem{Lan-ConstrainedStocasticProxMetNonconvex2019}
D.~Boob, Q.~Deng, and G.~Lan.
\newblock Stochastic first-order methods for convex and nonconvex functional
  constrained optimization.
\newblock {\em arXiv:1908.02734}, 2019.

\bibitem{florea2018accelerated}
M.~I. Florea and S.~A. Vorobyov.
\newblock An accelerated composite gradient method for large-scale composite
  objective problems.
\newblock {\em IEEE Transactions on Signal Processing}, 67(2):444--459, 2018.

\bibitem{MaxJeffRen-admm}
M.L.N. Gonçalves, J.G. Melo, and R.D.C. Monteiro.
\newblock Convergence rate bounds for a proximal admm with over-relaxation
  stepsize parameter for solving nonconvex linearly constrained problems.
\newblock {\em Pac. J. Optim.}, 15(3):379--398, 2019.

\bibitem{HongPertAugLag}
D.~Hajinezhad and M.~Hong.
\newblock Perturbed proximal primal–dual algorithm for nonconvex nonsmooth
  optimization.
\newblock {\em Math. Program.}, 176:207--245, 2019.

\bibitem{YheMoneiroNash}
Y.~He and R.D.C. Monteiro.
\newblock Accelerating block-decomposition first-order methods for solving
  composite saddle-point and two-player {N}ash equilibrium problems.
\newblock {\em SIAM J. Optim.}, 25(4):2182--2211, 2015.

\bibitem{YHe2}
Y.~He and R.D.C. Monteiro.
\newblock An accelerated {HPE}-type algorithm for a class of composite
  convex-concave saddle-point problems.
\newblock {\em SIAM J. Optim.}, 26(1):29--56, 2016.

\bibitem{Hiriart1}
J.B. Hiriart-Urruty and C.~Lemarechal.
\newblock {\em Convex Analysis and Minimization Algorithms I}.
\newblock Springer, Berlin, 1993.

\bibitem{Hiriart2}
J.B. Hiriart-Urruty and C.~Lemarechal.
\newblock {\em Convex Analysis and Minimization Algorithms II}.
\newblock Springer, Berlin, 1993.

\bibitem{ProxAugLag_Ming}
M.~Hong.
\newblock Decomposing linearly constrained nonconvex problems by a proximal
  primal dual approach: algorithms, convergence, and applications.
\newblock {\em arXiv:1604.00543}, 2016.

\bibitem{SZhang-Pen-admm}
B.~Jiang, T.~Lin, S.~Ma, and S.~Zhang.
\newblock Structured nonconvex and nonsmooth optimization algorithms and
  iteration complexity analysis.
\newblock {\em Comput. Optim. Appl.}, 72(3):115–157, 2019.

\bibitem{KongThesis2021}
W.~Kong.
\newblock Accelerated inexact first-order methods for solving nonconvex
  composite optimization problems.
\newblock {\em arXiv:2104.09685}, April 2021.

\bibitem{fistaReport2021}
W.~Kong, J.~G. Melo, and R.D.C. Monteiro.
\newblock {FISTA and Extensions - Review and New Insights}.
\newblock {\em Optimization Online}, 2021.

\bibitem{WJRproxmet1}
W.~Kong, J.G. Melo, and R.D.C. Monteiro.
\newblock Complexity of a quadratic penalty accelerated inexact proximal point
  method for solving linearly constrained nonconvex composite programs.
\newblock {\em SIAM J. Optim.}, 29(4):2566--2593, 2019.

\bibitem{WJRComputQPAIPP}
W.~Kong, J.G. Melo, and R.D.C. Monteiro.
\newblock An efficient adaptive accelerated inexact proximal point method for
  solving linearly constrained nonconvex composite problems.
\newblock {\em Comput. Optim. Appl.}, 76(2):305--346, 2019.

\bibitem{MinMax-RenWilliam}
W.~Kong and R.D.C. Monteiro.
\newblock An accelerated inexact proximal point method for solving
  nonconvex-concave min-max problems.
\newblock {\em arXiv:1905.13433v2}, 2019.

\bibitem{LanRen2013PenMet}
G.~Lan and R.D.C. Monteiro.
\newblock Iteration-complexity of first-order penalty methods for convex
  programming.
\newblock {\em Math. Program.}, 138(1):115--139, Apr 2013.

\bibitem{LanMonteiroAugLag}
G.~Lan and R.D.C. Monteiro.
\newblock Iteration-complexity of first-order augmented {L}agrangian methods
  for convex programming.
\newblock {\em Math. Program.}, 155(1):511--547, Jan 2016.

\bibitem{Li-Qu19}
F.~Li and Z.~Qu.
\newblock {An inexact proximal augmented Lagrangian framework with arbitrary
  linearly convergent inner solver for composite convex optimization}.
\newblock {\em arXiv:1909.09582}, 2019.

\bibitem{ImprovedShrinkingALM20}
Z.~Li, P.-Y. Chen, S.~Liu, S.~Lu, and Y.~Xu.
\newblock Rate-improved inexact augmented {L}agrangian method for constrained
  nonconvex optimization.
\newblock {\em Proc. 24th Int. Conf. Artif. Intell. and Statist.},
  130:170--2178, 2021.

\bibitem{HybridPenaltyAugLag19}
Z.~Li and Y.~Xu.
\newblock First-order inexact augmented {L}agrangian methods for convex and
  nonconvex programs: nonergodic convergence and iteration complexity.
\newblock {\em Preprint}, 2019.

\bibitem{PPmetNonconvex2019}
Q.~Lin, R.~Ma, and Y.~Xu.
\newblock Inexact proximal-point penalty methods for non-convex optimization
  with non-convex constraints.
\newblock {\em arXiv:1908.11518v4}, 2020.

\bibitem{pmlr-v32-lin14}
Q.~Lin and L.~Xiao.
\newblock An adaptive accelerated proximal gradient method and its homotopy
  continuation for sparse optimization.
\newblock {\em Proc. 31st Int. Conf. Mach. Learn.}, 32:73--81, 2014.

\bibitem{ShiqiaMaAugLag16}
Y.F. Liu, X.~Liu, and S.~Ma.
\newblock On the nonergodic convergence rate of an inexact augmented
  {L}agrangian framework for composite convex programming.
\newblock {\em Math. Oper. Res.}, 44(2):632--650, 2019.

\bibitem{zhaosongAugLag18}
Z.~{Lu} and Z.~{Zhou}.
\newblock {Iteration-complexity of first-order augmented {L}agrangian methods
  for convex conic programming}.
\newblock {\em arXiv:1803.09941}, 2018.

\bibitem{RJWIPAAL2020}
J.G. Melo, R.D.C. Monteiro, and H.~Wang.
\newblock Iteration-complexity of an inexact proximal accelerated augmented
  {L}agrangian method for solving linearly constrained smooth nonconvex
  composite optimization problems.
\newblock {\em arXiv:2006.08048}, 2020.

\bibitem{MontSvaiter_fista}
R.D.C. Monteiro, Ortiz, and Benar~F. Svaiter.
\newblock An adaptive accelerated first-order method for convex optimization.
\newblock {\em Comput. Optim. Appl.}, 64:31--73, 2016.

\bibitem{IterComplConicprog}
I.~Necoara, A.~Patrascu, and F.~Glineur.
\newblock Complexity of first-order inexact {L}agrangian and penalty methods
  for conic convex programming.
\newblock {\em Optim. Methods Softw.}, pages 1--31, 2017.

\bibitem{nesterov1983method}
Y.~E. Nesterov.
\newblock {\em Introductory lectures on convex optimization : a basic course}.
\newblock Kluwer Academic Publ., 2004.

\bibitem{nesterov2012gradient}
Y.E. Nesterov.
\newblock Gradient methods for minimizing composite functions.
\newblock {\em Math. Program.}, 140:1--37, 2013.

\bibitem{parikh2014proximal}
N.~Parikh and S.~Boyd.
\newblock Proximal algorithms.
\newblock {\em Foundations and Trends in optimization}, 1(3):127--239, 2014.

\bibitem{Patrascu2017}
A.~Patrascu, I.~Necoara, and Q.~Tran-Dinh.
\newblock Adaptive inexact fast augmented {L}agrangian methods for constrained
  convex optimization.
\newblock {\em Optim. Lett.}, 11(3):609--626, 2017.

\bibitem{Rockafellar70}
R.~T. Rockafellar.
\newblock {\em Convex Analysis}.
\newblock Princeton University Press, Princeton, 1970.

\bibitem{MR0418919}
R.~T. Rockafellar.
\newblock Augmented {L}agrangians and applications of the proximal point
  algorithm in convex programming.
\newblock {\em Math. Oper. Res.}, 1(2):97--116, 1976.

\bibitem{rockafellar1976augmented}
R.T. Rockafellar.
\newblock Augmented {L}agrangians and applications of the proximal point
  algorithm in convex programming.
\newblock {\em Mathematics of operations research}, 1(2):97--116, 1976.

\bibitem{inexactAugLag19}
M.~Sahin, A.~Eftekhari, A.~Alacaoglu, F.~Latorre, and V~Cevher.
\newblock An inexact augmented {L}agrangian framework for nonconvex
  optimization with nonlinear constraints.
\newblock {\em Adv. Neural Inf. Process. Syst.}, pages 13943--13955, 2019.

\bibitem{YangyangAugLag17}
Y.~Xu.
\newblock Iteration complexity of inexact augmented {L}agrangian methods for
  constrained convex programming.
\newblock {\em Math. Program.}, 185:199--244, 2021.

\bibitem{ErrorBoundJzhang-ZQLuo2020}
J.~Zhang and Z.-Q. Luo.
\newblock A global dual error bound and its application to the analysis of
  linearly constrained nonconvex optimization.
\newblock {\em arXiv:2006.16440}, 2020.

\bibitem{ADMMJzhang-ZQLuo2020}
J.~Zhang and Z.-Q. Luo.
\newblock A proximal alternating direction method of multiplier for linearly
  constrained nonconvex minimization.
\newblock {\em SIAM J. Optim.}, 30(3):2272--2302, 2020.

\end{thebibliography}

	\end{document}